\documentclass[leqno,twoside,11pt]{amsart}

\usepackage{latexsym,esint,url,amsmath, amssymb, mathrsfs, amsthm, cases, verbatim}

\theoremstyle{plain}
\def\endproof{\hspace*{\fill}\mbox{\ \rule{.1in}{.1in}}\medskip }

\newtheorem{theorem}{Theorem}[section]
\newtheorem{corollary}[theorem]{Corollary}
\newtheorem{lemma}[theorem]{Lemma}

\newtheorem{definition}[theorem]{Definition}

\theoremstyle{definition}

\newtheorem{remark}[theorem]{Remark}

\numberwithin{equation}{section}
\numberwithin{figure}{section}

\begin{document}

\title[The obstacle problem for the $p$-laplacian]
{The obstacle problem for the $p$-laplacian\\ via Optimal Stopping of Tug-of-War games}
\author{Marta Lewicka and Juan J. Manfredi}
\address{Marta Lewicka and Juan J. Manfredi, University of Pittsburgh, Department of Mathematics, 
139 University Place, Pittsburgh, PA 15260}
\email{lewicka@pitt.edu, manfredi@pitt.edu}

\date{\today}

\begin{abstract} We present a probabilistic  approach to the obstacle problem  for
the $p$-Laplace operator. The solutions are approximated 
by running processes determined by tug-of-war games plus noise,  and
letting the step size go to 
zero, not unlike the case when Brownian motion is  approximated by
random walks.  Rather than stopping the process when the boundary is
reached, the value function is  obtained  by maximizing over all
possible stopping times that are smaller than the exit time of the domain.   
\end{abstract}

\maketitle

\section{Introduction}
Let  $\mathfrak{L}$ be the second order differential operator:
$$\mathfrak{L}= \frac{1}{2} \, \textrm{trace}(\sigma\sigma'(x) D^2_x v),$$
whose matrix coefficient function $\sigma$ is Lipschitz  continuous. Consider the 
obstacle problem in $\mathbb{R}^N$:
\begin{equation} \label{obstaclelinear}
\min\left( - \mathfrak{L} v, v-g\right)=0.
\end{equation}
In order to solve (\ref{obstaclelinear})
probabilistically \cite{PH2}, one first solves the stochastic differential equation:
\begin{equation}\label{sdelinear}
d\mathbf{X}_t=  \sigma(\mathbf{X}_t)\,d\mathbf{W}_t,
\end{equation}
starting from $x$ at time $t=0$. Denoting its solution by $\{\mathbf{X}_t^x, t\ge 0\}$, the value
function is defined by taking the supremum over the set $\mathfrak{T}$ of all
stopping times valued in $[0,\infty]$:
\begin{equation}\label{valuelinear}
v(x)=\sup_{\tau\in\mathfrak{T}} \mathbb{E}\left[  
g(\mathbf{X}^x_\tau) \right],
\end{equation}
This value function,  under appropriate regularity hypothesis on $g$,
turns out then to be the unique
solution to (\ref{obstaclelinear}); for details see  Chapter 5 in \cite{PH2}.

The purpose of the present paper is to consider the analog problem when the
second order linear differential operator $\mathfrak{L}$ is replaced by the $p$-Laplacian:
\begin{equation}\label{introplap}
-\Delta_p u = -\mbox{div} \big(|\nabla u|^{p-2} \nabla u\big), \qquad
2\leq p <\infty.
\end{equation}
Since the operator $-\Delta_p $ is non-linear, we do not have a
suitable variant of the linear stochastic differential equation
(\ref{sdelinear}) that could be used to write a formula similar to
(\ref{valuelinear}).  Instead, we will show that one can use tug-of-war games with noise
as the basic stochastic process. More precisely, we will prove that  
the solutions to the obstacle problem for the $p$-Laplacian for
$p\in [2,\infty)$, can be interpreted as
limits of values of a specific discrete tug-of-war game with noise, when the step-size
$\epsilon$, 
determining the allowed
length of move of a token at each step of the game, converges to $0$. 

To explain our approach, let us first recall a concept of supersolutions suggested
by the tug-of-war characterization in  \cite{MPR0}. This notion is
based on the mean value properties, and it is equivalent to
the notion of viscosity supersolution in the class of continuous
functions.  Namely, let $\Omega\subset\mathbb{R}^N$ be an open, bounded set with Lipschitz
boundary and let $F:\partial{\Omega}\to\mathbb{R}$ be a Lipschitz
continuous boundary data. Choose the parameters $\alpha$ and $\beta$ as follows:
$$ \alpha=\frac{p-2}{N+p}, \qquad \beta =\frac{2+N}{N+p}, $$
 where $\alpha\ge 0$ since $p\ge 2$,  $\beta>0$, and $\alpha+\beta=1$.

\smallskip

\textbf{Supersolutions in the sense of means:} 
A continuous function $v:\Omega\to\mathbb{R}\cup\{\infty\}$ is a
supersolution  in the sense of means if 
whenever $\phi \in \mathcal{C}^{\infty}_0(\Omega)$ is such that $\phi(x)\le
v(x)$ for all $x\in \Omega$, with equality at one point  
$\phi(x_0)=v(x_0)$ ($\phi$ touches $v$ from below at $x_0$), then we have:
\begin{equation}\label{plapmean}
0\le - \phi (x_0) + {\displaystyle 
      \frac{\alpha}{2} \sup_{B_\epsilon(x_0)} \phi + \frac{\alpha}{2}
      \inf_{B_\epsilon(x_0)} \phi +\beta\fint_{B_\epsilon(x_0)} \phi }
    + o(\epsilon^2)  \qquad \mbox{ as } ~\epsilon\to 0^+.
\end{equation}

Above, by $0\le h(\epsilon) + o(\epsilon^2)$ as $\epsilon\to0^+$ we mean that:
$$\lim_{\epsilon\to0^+}\frac{\left[h(\epsilon)\right]^-}{\epsilon^2} =0.$$
Fixing a scale $\epsilon$, we now consider functions for which
(\ref{plapmean}) holds with equality, i.e. without the error term
$\epsilon^2$. 
Let $0<\epsilon_0\ll 1$ be a small constant and define the
fattened outer boundary set, together with the fattened domain: 
$$\Gamma = \{x\in \mathbb{R}^N\setminus\Omega;~
\mbox{dist}(x,\Omega)<\epsilon_0\},\qquad X=\Omega\cup\Gamma.$$

\smallskip

\textbf{$\epsilon$-$p$-harmonious functions:} 
Let $0< \epsilon\leq \epsilon_0$. A bounded, Borel function $u:
X\to\mathbb{R}$ is $\epsilon$-$p$-harmonious with boundary values  
$F:\bar{\Gamma}\to \mathbb{R}$  if:
\begin{equation}\label{dpp00}
u_\epsilon(x) = \left\{\begin{array}{ll}{\displaystyle   
      \frac{\alpha}{2} \sup_{B_\epsilon(x)} u_\epsilon + \frac{\alpha}{2}
      \inf_{B_\epsilon(x)} u_\epsilon +\beta\fint_{B_\epsilon(x)} u_\epsilon  } & \mbox{ for }
    x\in\Omega\vspace{2mm}\\
F(x) & \mbox{ for } x\in\Gamma.
\end{array}\right.
\end{equation}

\medskip

Then, it has been established in \cite{MPR} that $u=\lim_{\epsilon\to 0} u_\epsilon$ is a solution to the Dirichlet problem:
\begin{equation}\label{dirichlet}
\left\{\begin{array}{rcll}
-\Delta_p u &= &0\ \  &\mathrm{in}\ \Omega,\\
 u&=&F\ \  &\mathrm{   on }\  \partial\Omega.
\end{array}\right.
\end{equation}
Let now $\Psi:\mathbb{R}^n\to\mathbb{R} $ be a bounded, Lipschitz
function, which we assume to be compatible with the boundary data: $F(x)\geq
\Psi(x)$ for $x\in\partial{\Omega}$.  
The function $\Psi$ is interpreted as the obstacle and we consider the problem:
\begin{equation}\label{obstacle}
\left\{\begin{array}{rcll}
-\Delta_p u &\geq &0\ \  &\mathrm{in}\ \Omega,\\
 u&\geq &\Psi\ \ &\mathrm{   in }\  \Omega,\\
 -\Delta_p u & =  &0\ \ &\mathrm{   in }\  \{x\in\Omega; ~ u(x)>\Psi(x)\},\\
 u&=&F\ \  &\mathrm{   on }\  \partial\Omega.
\end{array}\right.
\end{equation}
That is, we want to find a $p$-superharmonic function $u$ which takes
boundary values $F$, which is above the obstacle $\Psi$, and which is
actually $p$-harmonic in the complement of the  contact set $\{x\in
\overline{\Omega}\colon u(x)=\Psi(x)\}$.  

The problem (\ref{obstacle}) has been extensively studied from the variational point of
view; see the seminal paper by Lindqvist  \cite{L} and the  book \cite{HKM}.  
In particular, regularity requirements for the domain $\Omega$, the
boundary data $F$ and the obstacle $\Psi$ can be vastly generalized. We
have, however, focused on the Lipschitz category for technical reasons in our proofs.  

\medskip

Our first result shows how to solve the obstacle problem using
$\epsilon$-$p$-superharmonious functions.
The dynamic programing principle (\ref{dpp}) below is similar to the Wald-Bellman equations of
optimal stopping (see Chapter 1 of \cite{PS2}).

\begin{theorem}\label{exists}
Let $\alpha\in [0,1)$ and $\beta=1-\alpha$.
Let $F:\bar\Gamma\rightarrow \mathbb{R}$ and $\Psi:
\mathbb{R}^N\rightarrow\mathbb{R}$ be two bounded, Borel functions
such that $\Psi\leq F$ in $\bar \Gamma$. 
Then there exists a unique bounded Borel function
$u:X\rightarrow\mathbb{R}$ which satisfies:
\begin{equation}\label{dpp}
u_\epsilon(x) = \left\{\begin{array}{ll}{\displaystyle \max\left\{\Psi(x),
      \frac{\alpha}{2} \sup_{B_\epsilon(x)} u_\epsilon + \frac{\alpha}{2}
      \inf_{B_\epsilon(x)} u_\epsilon +\beta\fint_{B_\epsilon(x)} u_\epsilon \right\}} & \mbox{ for }
    x\in\Omega\vspace{2mm}\\
F(x) & \mbox{ for } x\in\Gamma.
\end{array}\right.
\end{equation}
\end{theorem}

After proving Theorem \ref{exists} in Section \S \ref{rm}, we proceed to
establishing that $u=\lim_{\epsilon\to0^+} u_\epsilon$ is the
solution to the obstacle problem (\ref{obstacle}). The key step is to
show that $\{u_\epsilon\}$ is equicontinuous up to scale $\epsilon$,
i.e. the functions $u_\epsilon$ may be discontinuous but the
discontinuities are of size roughly $\epsilon$. Then, an  easy
extension of the Arzel\'a-Ascoli theorem from \cite{MPR} shows that
there are subsequences of $\{u_\epsilon\}$ that converge uniformly to
a function $u$. The standard stability of viscosity solutions yields
then that any such limit is a viscosity solution of
(\ref{obstacle}). By uniqueness (see Lemma \ref{uni} and its proof in
the Appendix) they all must agree, and we obtain:

\begin{theorem}\label{introthm}
Let $p\in [2,\infty)$ and let $u_\epsilon:\Omega\cup\Gamma\to\mathbb{R}$ be the unique
$\epsilon$-$p$-superharmonious function solving (\ref{dpp}) with $\alpha =
\frac{p-2}{p+N}$ and $\beta = \frac{2+N}{p+N}$.
Then $u_\epsilon$ converge as $\epsilon \to 0$,
uniformly in $\bar\Omega$, to a continuous function $u$ which is the
unique viscosity solution to the obstacle problem (\ref{obstacle}).
\end{theorem}

Towards the proof, the key estimate (in Lemma \ref{AA}) bounds the oscillation of $u_\epsilon$
in a uniform way. Namely, given $\eta>0$ there are small $r_0, \epsilon_0>0$
so that whenever $\epsilon < \epsilon_0$, then for all  $x_0, y_0\in \overline{\Omega}$:
\begin{equation}\label{key1}
|x_0-y_0|<r_0 \implies |u_\epsilon(x_0) - u_\epsilon(y_0)|<\eta.
\end{equation}
To deduce (\ref{key1}) we use probabilistic techniques, interpreting
$u_\epsilon$ as value functions of certain tug-of-war
games and finding an appropriate extension of  
(\ref{valuelinear}).  In Section \ref{prob2} we present the details of
this construction, involving stochastic processes  (tug-of-war games with
noise) needed to write down the representation formulas for
$u_\epsilon$. For the case of linear equations (that correspond to $p=2$) with
variable coefficients, a similar version of the representation
formula (\ref{value}) below is due to Pham \cite{PH} and \O ksendal-Reikvam
\cite{OR}.   Note that since $\epsilon>0$ is fixed, we omit it in the
statement that follows.  

\begin{theorem}\label{prob}
Let $\alpha, \beta\geq 0$ satisfy $\alpha + \beta=1$.
Let $F:\Gamma\rightarrow \mathbb{R}$ and $\Psi:
\mathbb{R}^N\rightarrow\mathbb{R}$ be two bounded, Borel functions
such that $\Psi\leq F$ in $\Gamma$. Define: 
$$G:X\rightarrow \mathbb{R} \qquad G = \chi_{\Gamma} F
+ \chi_{\Omega} \Psi,$$
where $\chi_{A}$ stands for the characteristic function of a set
$A\subset X$. Define the two value functions:
\begin{equation}\label{value}
u_I(x_0) = \sup_{\tau,\sigma_{I}} \inf_{\sigma_{II}}
\mathbb{E}^{x_0}_{\tau,\sigma_I, \sigma_{II}} [G\circ x_\tau], \qquad 
u_{II}(x_0) = \inf_{\sigma_{II}} \sup_{\tau,\sigma_{I}} 
\mathbb{E}^{x_0}_{\tau,\sigma_I, \sigma_{II}} [G\circ x_\tau],
\end{equation}
where $\sup$ and $\inf$ are taken over all strategies $\sigma_I$,
$\sigma_{II}$ and stopping times $\tau\leq \tau_0$ that do not
superseed the exit time $\tau_0$ from the set $\Omega$. Then:
$$u_I = u = u_{II} \quad \mbox{ in } \Omega,$$
where $u$ is a bounded, Borel function satisfying (\ref{dpp}).
\end{theorem}

The core of this paper can be found in Section \S \ref{bo2}, where we use the
representation formulas from Theorem \ref{prob} to establish the
oscillation estimate (\ref{key1}). We provide full details of the
proof for the equicontinuity estimates 
in Section \S 4, the proof of Theorem \ref{introthm} in Section \S 5, and
the fact that our games end almost surely in Section \S 6.  

\bigskip

We finish this introduction by discussing other notions of solutions 
for (\ref{introplap}) in addition to (\ref{plapmean}):

\begin{itemize}
\item[i)] \textbf{Weak (or Sobolev) supersolutions:} These are
  functions $v\in W^{1,p}_\textrm{loc}(\Omega)$ such that: 
$$\int_{\Omega} \langle|\nabla v|^{p-2} \nabla v, \nabla \phi\rangle~\mbox{d}x \ge 0$$ 
for all test functions $\phi \in \mathcal{C}^{\infty}_0(\Omega, \mathbb{R}_+)$ that are
non-negative in $\Omega$.  
\smallskip
\item[ii)] \textbf{Potential theoretic supersolutions or
    $p$-superharmonic functions:} A lower-semi\-con\-ti\-nuous function
  $v:\Omega\to\mathbb{R}\cup\{\infty\}$ is $p$-superharmonic
  if it is not identically $\infty$ on any connected component of
  $\Omega$ and it satisfies the comparison principle with respect to
  $p$-harmonic functions, that is: if $D\Subset\Omega$, and $w\in
  \mathcal{C}(\bar{D})$ is $p$-harmonic in $D$ satisfying $w\le v$ on
  $\partial{D }$, then we must have:  $w\le v$ on $D$.  
\smallskip
\item[iii)] \textbf{Viscosity supersolutions:}  A lower-semicontinuous
  function $v:\Omega\to\mathbb{R}\cup\{\infty\}$ is a
  viscosity $p$-supersolution if it is not identically $\infty$ on any
  connected component of $\Omega$,  and if
whenever $\phi \in \mathcal{C}^{\infty}_0(\Omega)$ is such that $\phi(x)\le
v(x)$ for all $x\in \Omega$ with equality at one point  
$\phi(x_0)=v(x_0)$ ($\phi$ touches $v$ from below at $x_0$), and
$\nabla\phi(x_0)\not=0$, then we have:
$$ -\Delta_p \phi (x_0) \ge 0.$$
\end{itemize}

\medskip

The fact that weak supersolutions are potential theoretic and
viscosity supersolutions follows from the comparison principle and a
regularity argument implying the lower-semicontinuity; see for example
Chapter 3 in \cite{HKM}. 
The fact that bounded $p$-superharmonic functions are weak
supersolutions was established by Lindqvist in \cite{L}.\footnote{In
  fact, this is also the first reference that we have been able to
  locate for the classical case  $p=2$.} Note that an
arbitrary, not necessarily bounded $p$-superharmonic function $v$ is
always the pointwise increasing limit of bounded $p$-superharmonic
functions $v_n=\min\{v, n\}$.  The equivalence between viscosity
supersolutions and $p$-superharmonic functions was established in 
\cite{JLM}. Therefore, the three notions of supersolution agree on the
class of bounded functions; see also the Appendix where for
completeness we present the folklore argument stating that viscosity solutions are
weak solutions and thus they are unique.

\medskip 

Finally, we recall a classical useful observation. When $u\in \mathcal{C}^2$
and $\nabla u(x)\neq 0$, then one can express  the $p$-Laplacian as a
combination of the ordinary Laplacian and the $\infty$-Laplacian:
\begin{equation}\label{plapla} 
\Delta_p u (x)= |\nabla u|^{p-2}\big(\Delta u (x)+ (p-2)\Delta_\infty u(x)\big),
\end{equation}
where:
\begin{equation}\label{inflapla}
\Delta_\infty u(x) = \Big\langle \nabla^2 u(x) \frac{\nabla u(x)}{|\nabla
u(x)|}, \frac{\nabla u(x)}{|\nabla u(x)|}\Big\rangle. 
\end{equation}
The tug-of-war interpretation of the $\infty$-Laplacian has been
developed in the fundamental paper \cite{PSSW}. The obstacle problem
for (\ref{inflapla}) has been studied in \cite{MRS}. A similar
treatment as in the present paper, for the double obstacle problem,
has been developed in \cite{CLM}.

\bigskip

\noindent{\bf Acknowledgments.}
The first author was partially supported by  NSF awards DMS-0846996
and DMS-1406730. The second author was partially supported by NSF award DMS-1001179.

\section{$\epsilon$-$p$-superharmonious functions: a proof of Theorem \ref{exists}}\label{rm}

The proof uses the monotonicity arguments of the Perron method as
extended by  \cite{LPS}, modified to accommodate the obstacle
constraint. 

\smallskip

{\bf 1.}  The solution to (\ref{dpp}) will be obtained as the uniform limit of
iterations $u_{n+1} = Tu_n$, where, for any bounded Borel function
$v:X\rightarrow \mathbb{R}$, we define:
\begin{equation}\label{dppT}
Tv(x) = \left\{\begin{array}{ll}{\displaystyle \max\left\{\Psi(x),
      \frac{\alpha}{2} \sup_{B_\epsilon(x)} v + \frac{\alpha}{2}
      \inf_{B_\epsilon(x)} v +\beta\fint_{B_\epsilon(x)} v \right\}} & \mbox{ for }
    x\in\Omega\vspace{2mm}\\
v(x) & \mbox{ for } x\in\Gamma
\end{array}\right.
\end{equation}
and where we put:
\begin{equation}\label{dpp0}
u_0 = \chi_{\Gamma} F + \chi_\Omega \left(\inf_X\Psi\right).
\end{equation}

We easily note that $u_1 = Tu_0\geq u_0$ in
$X$. Consequently: $u_2=Tu_1\geq Tu_0 = u_1$ in $X$ and, by induction,
the sequence of Borel functions $\{u_n\}$ is nondecreasing in $X$. 
Also, $\{u_n\}_{n=1}^\infty$ satisfies:
$$ \Psi\leq u_n \leq \max\left\{\sup_{\Gamma} F, \sup_X \Psi\right\} \quad
\mbox{in } \Omega,$$
and clearly $u_n = F$ on $\Gamma$. Therefore, the sequence
converges pointwise to a bounded Borel function
$u:X\rightarrow\mathbb{R}$, satisfying: $u_{|\Gamma} = F$.

\smallskip

{\bf 2.} We now show that the convergence of $\{u_n\}$ to $u$ is
uniform in $X$. This will automatically imply that $u=\lim Tu_n =
T(\lim u_n) = Tu$, and hence yield the desired existence result.
We argue by contradiction and assume that:
$$M = \lim_{n\to\infty} \sup_{x\in X} (u-u_n)(x) > 0.$$
Fix a small $\delta>0$ and take $n>1$ so that:
$$ \sup_X ~(u-u_n) < M+\delta \quad \mbox{ and } \quad \forall x\in\Omega \quad
\beta\fint_{B_\epsilon(x)} u-u_n \leq \frac{\beta}{|B_\epsilon(x)|}
\int_X u-u_n < \delta.$$
The second condition above is justified by the monotone convergence
theorem.

Take now $x_0\in\Omega$ satisfying: $u(x_0) -
u_{n+1}(x_0)>M-\delta$. Note that if $u(x_0) = \Psi(x_0)$ then since
$u_j(x_0)$ increases to $u(x_0)$ and all $u_j(x_0)\geq \Psi(x_0)$, there
would be $u_n(x_0) = \Psi(x_0)$. Therefore $u(x_0)>\Psi(x_0)$ and
consequently:
$$\exists m>n \quad u_{m+1}(x_0) - u_{n+1}(x_0) > M-2\delta \quad 
\mbox{ and } \quad  u_{m+1}(x_0) > \Psi(x_0).$$
We now compute:
\begin{equation}\label{calc}
\begin{split}
M-2\delta & < u_{m+1}(x_0) - u_{n+1}(x_0) \\ & = \frac{\alpha}{2}
\left(\sup_{B_\epsilon(x_0)} u_m - \sup_{B_\epsilon(x_0)} u_n\right) + 
\frac{\alpha}{2} \left(\inf_{B_\epsilon(x_0)} u_m -
  \inf_{B_\epsilon(x_0)} u_n\right) +
\beta\fint_{B_\epsilon(x_0)} u_m - u_n \\ & \leq \alpha 
\sup_{B_\epsilon(x_0)} (u_m - u_n) 
+ \beta\fint_{B_\epsilon(x_0)} u_m - u_n
\leq \alpha \sup_{B_\epsilon(x_0)} (u - u_n) 
+ \beta\fint_{B_\epsilon(x_0)} u - u_n \\ & < \alpha (M+\delta) + \delta =
\alpha M + (\alpha+1)\delta.
\end{split}
\end{equation}
This implies that $M<\alpha M + (\alpha+3)\delta$, which
clearly is a contradiction for $\delta$ sufficiently small,
since $\alpha<1$.

\smallskip

{\bf 3.} We now prove uniqueness of solutions to (\ref{dpp}). Assume,
by contradiction, that $u$ and $\bar u$ are two distinct solutions and denote:
$$M=\sup_\Omega (u-\bar u) > 0.$$
Let $\{x_n\}_{n\geq 1}$ be a sequence of points in $\Omega$ such that
$\lim_{n\to\infty} (u-\bar u)(x_n) = M$. Observe that for large $n$ there must be:
$u(x_n) > \Psi(x_n)$, since $M>0$. Therefore, as in (\ref{calc}), we get:
\begin{equation*}
\begin{split}
(u - \bar u)(x_n) & = \frac{\alpha}{2}
\left(\sup_{B_\epsilon(x_n)} u - \sup_{B_\epsilon(x_n)} \bar u\right) + 
\frac{\alpha}{2} \left(\inf_{B_\epsilon(x_n)} u -
  \inf_{B_\epsilon(x_n)} \bar u\right) +
\beta\fint_{B_\epsilon(x_n)} u - \bar u \\ & \leq \alpha 
\sup_{B_\epsilon(x_n)} (u - \bar u) 
+ \beta\fint_{B_\epsilon(x_n)} u - \bar u
\leq \alpha M + \beta\fint_{B_\epsilon(x_n)} u - \bar u.
\end{split}
\end{equation*}
Passing to the limit with $n\to\infty$, where $\lim x_n = x_0$, we obtain:
$$ M\leq \alpha M +\beta \fint_{B_\epsilon (x_0)} u-\bar u,$$
and hence: $M\leq \fint_{B_\epsilon(x_0)} u-\bar u, $ since
$\beta>0$. Consequently:
$u-\bar u=M$ almost everywhere in $B_\epsilon(x_0)$, and hence in
particular the following set is nonempty:
$$G=\{x\in X;~ (u-\bar u)(x) = M\} \neq \emptyset.$$
By the same argument as above, we see that in fact for all $x\in G$, the set 
$B_\epsilon(x) \setminus G$ has measure $0$. We conclude that: 
$$u-\bar u = M \qquad \mbox{a.e. in } X $$
which contradicts the fact that $G\cap\Gamma=\emptyset$, and
proves the result.
\endproof

\medskip

We further easily derive the following weak comparison principle:

\begin{corollary}\label{comparison}
Let $u$ and $\bar u$ be the unique solutions to (\ref{dpp}) with the
respective boundary data $F$ and $\bar F$ and obstacle constraints
$\Psi$ and $\bar\Psi$, satisfying the assumptions of Theorem
\ref{exists}. If $F\leq \bar F$ and $\Psi\leq \bar\Psi$ then $u\leq
\bar u$ in $\Omega$. 
\end{corollary}

\begin{proof}
Let $\{u_n\}$ and $\{\bar u_n\}$ be the approximating sequences for
$u$ and $\bar u$, as in the proof of Theorem \ref{exists}. By
(\ref{dpp0}): $u_0\leq \bar u_0$, which results in $u_n\leq \bar u_n$
for every $n$, in view of (\ref{dpp}). Consequently, the limits $u$
and $\bar u$ satisfy the same pointwise inequality.
\end{proof}

\section{Game-theoretical interpretation of the
  $\epsilon$-$p$-superharmonious functions}\label{prob2}

We now link the $\epsilon$-$p$-superharmonious function $u$ solving (\ref{dpp}) to the probabilistic
setting.  We define this setting in detail, as this paper is dedicated to
analysts rather than probabilists. All the basic concepts can be found
in the classical textbook \cite{Var}.

\subsection{The measure spaces $(X^{\infty, x_0},
  \mathcal{F}_n^{x_0})$ and $(X^{\infty, x_0},  \mathcal{F}^{x_0})$.} \label{start}
Fix any $x_0\in X$ and consider the
space of infinite sequences $\omega$ (recording positions of token during the game),
starting at $x_0$:
$$X^{\infty, x_0} = \{\omega=(x_0, x_1, x_2\ldots); ~x_n\in X \mbox{
  for all } n\geq 1\}.$$
For each $n\geq 1$, let $\mathcal{F}_n^{x_0}$ be the $\sigma$-algebra
of subsets of $X^{\infty, x_0}$, containing sets of the form:
$$A_1\times \ldots \times A_n := \{\omega\in X^{\infty, x_0}; ~ x_i\in 
A_i \mbox{ for } i:1\ldots n\}, $$
for all $n$-tuples of Borel sets $A_1,\ldots, A_n\subset X$. Although the expression in the left
hand side above is, formally, a Borel subset of $\mathbb{R}^{Nn}$, we will, with a slight
abuse of notation, identify it with the set of infinite histories
$\omega$ with completely undetermined positions beyond $n$.
Let $\mathcal{F}^{x_0}$ be now defined as the smallest $\sigma$-algebra of subsets of
$X^{\infty, x_0}$, containing $\bigcup_{n=1}^\infty \mathcal{F}_n^{x_0}$. Clearly, 
the increasing sequence $\{\mathcal{F}_n^{x_0}\}_{n\geq 1}$ is a
filtration of $\mathcal{F}^{x_0}$, and the coordinate projections
$x_n(\omega) = x_n$ are $\mathcal{F}^{x_0}$ measurable (random
variables) on $X^{\infty, x_0}$.

\subsection{The stopping times $\tau_0$ and $\tau$.} 
Define the exit time from the set $\Omega$:
$$\tau_0(\omega) = \min\{n\geq 0; ~ x_n\in\Gamma\}$$
where we adopt the convention that the minimum over the empty set equals
$+\infty$. This way: $\tau_0: X^{\infty, x_0}\to \mathbb{N}\cup
\{+\infty\}$ is $\mathcal{F}^{x_0}$ measurable and, in fact, it is a
stopping time with respect to the filtration
$\{\mathcal{F}_n^{x_0}\}$, that is:
\begin{equation}\label{stop}
\forall n\geq 0 \qquad \{\omega\in X^{\infty, x_0}; ~
\tau_0(\omega)\leq n\} \in \mathcal{F}_n^{x_0}.
\end{equation}
Let now $\tau: X^{\infty, x_0}\to \mathbb{N}\cup \{+\infty\}$ be any
stopping time (i.e. a random variable satisfying (\ref{stop}), where
$\tau_0$ is replaced by $\tau$) such that:
\begin{equation}\label{stop2}
\tau\leq \tau_0.
\end{equation}
For $n\geq 1$ we define the Borel sets:
$$A_n^\tau = \{(x_0, x_1, \ldots, x_n); ~\exists
\omega = (x_0, x_1, \dots, x_n, x_{n+1},\ldots)\in X^{\infty, x_0},
\quad \tau(\omega)\leq n\}. $$
By (\ref{stop2}), it follows that $(x_0,\ldots, x_n)\in
A_n^\tau$ whenever $x_n\in\Gamma$.

\subsection{The strategies $\sigma_{I}$, $\sigma_{II}$.} For every $n\geq 1$,
let $\sigma_I^n, \sigma_{II}^n:X^{n+1}\to X$ be Borel measurable
functions with the property that:
$$ \sigma_I^n(x_0, x_1,\ldots, x_n), ~\sigma_{II}^n(x_0, x_1,\ldots, x_n)\in  B_\epsilon (x_n)\cap X. $$
 We call $\sigma_I = \{\sigma_I^n\}_{n\geq 1}$ and $\sigma_{II} =
 \{\sigma_{II}^n\}_{n\geq 1}$ the strategies of Players I and II, respectively.

\subsection{The probability measure
  $\mathbb{P}^{x_0}_{\tau,\sigma_{I}, \sigma_{II}}$.} \label{probab}
Fix two parameters $\alpha, \beta\geq0$, such that: 
$\alpha+\beta=1$. Given $\tau, \sigma_{I}, \sigma_{II}$ as above, we
define now a family of probabilistic (Borel) measures on $X$,
parametrised by the finite histories $(x_0,\ldots, x_n)$:
\begin{equation}\label{gamma}
\begin{split}
& \forall n\geq 1 \quad \forall x_1, \ldots, x_n\in X\qquad  
\gamma_n[x_0, x_1, \ldots, x_n] = \\ & \qquad = \left\{ \begin{array}{ll}
\displaystyle{\frac{\alpha}{2} \delta_{\sigma_I^n(x_0, x_1, \ldots, x_n)} +
\frac{\alpha}{2} \delta_{\sigma_{II}^n(x_0, x_1, \ldots, x_n)} + 
\beta \frac{\mathcal{L}_N\lfloor B_\epsilon(x_n)}{|B_\epsilon(x_n)|} }
& \mbox{ when } (x_0,  \ldots, x_n)\not\in A_n^\tau\vspace{2mm}\\ 
\delta_{x_n} & \mbox{ otherwise}
\end{array}\right.
\end{split}
\end{equation}
where $\delta_y$ denotes the Dirac delta at a given $y\in X$, while
the measure multiplied by $\beta$ above stands for the $N$-dimensional
Lebesgue measure restricted to the ball $B_\epsilon(x_n)$ and
normalized by the volume of this ball. Note that since $\tau\leq
\tau_0$, then $\gamma_n[x_0, x_1, \ldots, x_n] = \delta_{x_n}$
whenever $x_n\in\Gamma$.

For every $n \geq 1$ we now define the probability measure 
$\mathbb{P}^{n,x_0}_{\tau,\sigma_{I}, \sigma_{II}}$ on $(X^{\infty, x_0},
  \mathcal{F}_n^{x_0})$ by setting:
\begin{equation}\label{partial}
\mathbb{P}^{n, x_0}_{\tau,\sigma_{I}, \sigma_{II}} (A_1\times \ldots
\times A_n) = \int_{A_1} \ldots \int_{A_n} 1 ~\mbox{d}\gamma_{n-1}[x_0,
x_1, \ldots, x_{n-1}] \ldots \mbox{d}\gamma_{0}[x_0]  
\end{equation}
for every $n$-tuple of Borel sets $A_1,\ldots, A_n\subset X$. 
Here, $A_1$ is interpreted as the set of possible successors $x_1$ of
the initial position $x_0$, which we integrate
$\mbox{d}\gamma_{0}[x_0]$, while $x_n\in A_n$ is a possible 
successor of $x_{n-1}$ which we integrate $\mbox{d}\gamma_{n-1}[x_0,
x_1, \ldots, x_{n-1}]$, etc. 
The following observation justifies the definition (\ref{partial}): 
\begin{lemma}
The family $\{\gamma_n[x_0, x_1, \ldots, x_n] \}$ in (\ref{gamma}) has the following
measurability property. For every $n\geq 1$ and every Borel set
$A\subset X$, the function:
$$X^{n+1}\ni (x_0, x_1, \ldots, x_n) \mapsto \gamma_n[x_0, x_1, \ldots,
x_n](A)\in\mathbb{R}$$
is Borel measurable.
\end{lemma}

\smallskip

It is clear that the family $\{\mathbb{P}^{n, x_0}_{\tau,\sigma_{I}, \sigma_{II}}\}_{n\geq 1}$ is
consistent (see \cite{Var}), with the transition probabilities $\gamma_n[x_0, x_1, \ldots,
x_n]$. Consequently, in virtue of the Kolmogoroff's Consistency Theorem, it
generates uniquely the probability measure
$\mathbb{P}^{x_0}_{\tau,\sigma_{I},
  \sigma_{II}}=\lim_{n\to\infty}\mathbb{P}^{n, x_0}_{\tau,\sigma_{I},
  \sigma_{II}}$ on $(X^{\infty, x_0}, \mathcal{F}_n)$ so that:
$$\forall n\geq 1\quad \forall A_1\times \ldots \times A_n\in
\mathcal{F}_n^{x_0} 
\qquad \mathbb{P}^{x_0}_{\tau,\sigma_{I},
  \sigma_{II}} (A_1\times \ldots \times A_n) 
= \mathbb{P}^{n, x_0}_{\tau,\sigma_{I}, \sigma_{II}}  (A_1\times \ldots \times A_n).$$ 
One can easily prove the following useful observation, which follows
by directly checking the definition of conditional expectation:
\begin{lemma}\label{condi}
Let $v:X\rightarrow \mathbb{R}$ be a bounded Borel function. For any $n\geq 1$,
the conditional expectation $\mathbb{E}^{x_0}_{\tau,\sigma_{I},
  \sigma_{II}}\{v\circ x_n\mid \mathcal{F}^{x_0}_{n-1}\}$
of the random variable $v\circ x_n$  is a  $\mathcal{F}^{x_0}_{n-1}$
measurable function on $X^{\infty, x_0}$ (and hence it depends only on
the initial $n$ positions in the history $\omega =
(x_0, x_1,\ldots )\in X^{\infty, x_0}$), given by:
$$\mathbb{E}^{x_0}_{\tau,\sigma_{I},
  \sigma_{II}}\{v\circ x_n\mid \mathcal{F}^{x_0}_{n-1}\} (x_0, \ldots,x_{n-1}) 
= \int_X v ~\mbox{d}\gamma_{n-1}[x_0,\ldots, x_{n-1}].$$
\end{lemma}

We also have:
\begin{lemma}\label{lemkoniec}
In the above setting, assume that $\beta>0$. Then  each game stops almost surely, i.e.:
\begin{equation}\label{koniec}
\mathbb{P}^{x_0}_{\tau,\sigma_I,\sigma_{II}} \big(\{\tau<\infty\}\big) = 1.
\end{equation}
\end{lemma}
For convenience of the reader, we give a self-contained proof of this
observation in the Appendix.

\subsection{$\epsilon$-$p$-superharmonious functions and game values.}

Before proving Theorem \ref{prob}, we need a lemma on almost optimal
selections. This lemma was put forward in \cite{LPS} and now we
present its possible elementary proof.

\begin{lemma}\label{selection}
Let $u:X\to\mathbb{R}$ be a bounded, Borel function. Fix $\delta,
\epsilon>0$. There exist Borel functions $\sigma_{sup},
\sigma_{inf}:\Omega\to X$ such that:
\begin{equation}\label{balls}
\forall x\in\Omega \qquad \sigma_{sup}(x), \sigma_{inf}(x) \in B_\epsilon(x)
\end{equation}
and:
\begin{equation}\label{req2}
\forall x\in\Omega \qquad u(\sigma_{sup}(x)) \geq
\sup_{B_\epsilon(x)} u - \delta, \qquad 
u(\sigma_{inf}(x)) \leq \inf_{B_\epsilon(x)} u + \delta.
\end{equation}
\end{lemma}
\begin{proof}
We will prove existence of $\sigma_{sup}$, while existence of
$\sigma_{inf}$ follows in a similar manner. 

{\bf 1.} Let $u=\chi_{A}$ for some Borel set $A\subset X$. 
Without loss of generality $\delta<\frac{1}{3}$. 
We write $A+B_\delta(0) = \bigcup_{i=1}^\infty B_\delta (x_i)$ as the union of
countably many open balls, and define:
$$\forall x\in\Omega \qquad \sigma_{sup}(x) = \left\{ \begin{array}{ll}
x & \mbox{ if } x\not\in A+B_\delta(0)\\
x_i & \mbox{ if } x\in B_\delta(x_i) \setminus \bigcup_{j=1}^{i-1}B_\delta(x_j)
\end{array}\right.$$
Clearly, $\sigma_{sup}$ above is Borel as a pointwise limit of Borel
functions. 

{\bf 2.} Let $u=\sum_{k=1}^n\alpha_k\chi_{A_k}$ be a simple function,
given by disjoint Borel sets $A_k\subset X$ and $\alpha_1 <
\alpha_2<\ldots < \alpha_n$. Without loss of generality $\delta<\min_{k=1\ldots
  n-1}\frac{\alpha_{k+1} - \alpha_k}{3}$. We now write, as before:
$A_k+B_\delta(0) = \bigcup_{i=1}^\infty B_\delta (x_i^k)$, and we
subsequently define:
$$\forall x\in\Omega \qquad \sigma_{sup}(x) = \left\{ \begin{array}{ll}
x & \mbox{ if } x\not\in (\bigcup_{k=1}^n A_k)+B_\delta(0)\\
x_i^k & \mbox{ if } x\in B_\delta(x_i^k) \setminus
\left(\bigcup_{j=1}^{i-1}B_\delta(x_j^k) \cup \bigcup_{j>k} (A_j + B_\delta(0))\right)
\end{array}\right.$$

{\bf 3.} In  the  general case when $u$ is an arbitrary bounded Borel function,
consider a simple function   $v$ such that $\|u-v\|_{L^\infty(X)}\leq
\frac{\delta}{3}$. By the previous construction, there exists
$\sigma_{sup}:\Omega\to X$ which is a
sup-selection for $v$, with the error $\frac{\delta}{3}$. Then we
have: 
$$\forall x\in\Omega \qquad u(\sigma_{sup}(x)) \geq v(\sigma_{sup}(x))
- \frac{\delta}{3} \geq \sup_{B_\epsilon(x)} v -\frac{2\delta}{3} 
\geq \sup_{B_\epsilon(x)} u - \delta,$$
and therefore $\sigma_{sup}$ is 
also the required sup-selection for the function $u$.
\end{proof}

\begin{remark}\label{remi}
If we replace the open balls
$B_\epsilon(x)$ in the requirement (\ref{balls}) by the closed ones,
then the Borel selection satisfying (\ref{req2}) may not exist.
Take $\epsilon=1, \delta=\frac{1}{3}$ and let $u=\chi_A$ where
$A\subset \mathbb{R}^3$ is a bounded Borel set with the property that
$A+\bar B_1(0)$ is not a Borel set. The existence of $A$ is nontrivial
(see \cite{LPS}) and relies on the existence of a $2$d Borel set whose
projection on the $x_1$ axis is not Borel. This result extends the
famous construction of Erdos and Stone \cite{ES} of a compact (Cantor) set $A$
and a $G_\delta$ set $B$ such that $A+B$ is not Borel.
\end{remark}

\medskip

\noindent {\bf Proof of Theorem \ref{prob}.}

{\bf 1.}  We first show that:
\begin{equation}\label{uII<u}
u_{II} \leq u \qquad \mbox{in } \Omega.
\end{equation}
Fix $\eta>0$ and fix any strategy $\sigma_{I}$ and a stopping time $\tau\leq \tau_0$.
By Lemma \ref{selection}, there exists a (Markovian) strategy $\sigma_{0, II}$ for
Player II, such that $\sigma_{0, II}^{n}(x_0, \ldots x_n) =
\sigma_{0, II}^{n}(x_n)$ and that:
\begin{equation}\label{almost}
\begin{split}
\forall n\geq 1\quad \forall x_n \in X \qquad
u(\sigma_{0, II}^{n}(x_n)) \leq \inf_{B_\epsilon(x_n)} u
+ \frac{\eta}{2^{n+1}}
\end{split}
\end{equation}
Using Lemma \ref{condi}, definition (\ref{gamma}), 
suboptimality in (\ref{almost}) and the equation (\ref{dpp}), we compute:
\begin{equation*}
\begin{split}
 \forall (x_0, \ldots, x_{n-1})&\not\in A^\tau_{n-1} \qquad
\mathbb{E}^{x_0}_{\tau,\sigma_{I},
  \sigma_{0, II}}\{u\circ x_n + \frac{\eta}{2^n}\mid \mathcal{F}^{x_0}_{n-1}\} (x_0, \ldots,x_{n-1}) 
\\ & = \int_X u ~\mbox{d}\gamma_{n-1}[x_0,\ldots, x_{n-1}] + \frac{\eta}{2^n}\\
& = \frac{\alpha}{2} u(\sigma_I^{n-1}(x_0, \ldots, x_{n-1})) +
\frac{\alpha}{2} u(\sigma_{0, II}^{n-1}(x_0, \ldots, x_{n-1}))  
+ \beta\fint_{B_\epsilon(x_{n-1})} u + \frac{\eta}{2^n} \\
& \leq \frac{\alpha}{2} \sup_{B_\epsilon(x_{n-1})} u + \frac{\alpha}{2} 
\inf_{B_\epsilon(x_{n-1})} u + 
\beta\fint_{B_\epsilon(x_{n-1})} u +
\frac{\eta}{2^n}(1+\frac{\alpha}{2})\\
& \leq u(x_{n-1}) + \frac{\eta}{2^{n-1}} = \Big(u\circ x_{n-1} +
\frac{\eta}{2^{n-1}}\Big) (x_0, \ldots,x_{n-1}).
\end{split}
\end{equation*}
On the other hand, when $(x_0, \ldots, x_{n-1})\in A^\tau_{n-1}$, then $\mathbb{E}^{x_0}_{\tau,\sigma_{I},
  \sigma_{0, II}}\{u\circ x_n + \frac{\eta}{2^n}\mid \mathcal{F}^{x_0}_{n-1}\} (x_0, \ldots,x_{n-1}) 
=  u(x_{n-1}) + \frac{\eta}{2^{n}}$ directly from Lemma
\ref{condi} and by (\ref{gamma}).
We therefore obtain that the sequence of random variables $\{ u\circ
x_n + \frac{\eta}{2^n}\}_{n\geq 0}$ is a supermartingale with respect
to the filtration $\{\mathcal{F}^{x_0}_{n}\}$. It follows that:
\begin{equation*}
\begin{split}
 u_{II}(x_0) &\leq \sup_{\tau, \sigma_{I}} \mathbb{E}^{x_0}_{\tau,\sigma_{I},
  \sigma_{0, II}}[G\circ x_\tau + \frac{\eta}{2^{\tau}}] \leq 
\sup_{\tau, \sigma_{I}} \mathbb{E}^{x_0}_{\tau,\sigma_{I},
  \sigma_{0, II}}[u\circ x_\tau + \frac{\eta}{2^{\tau}}] \\ & \leq
\sup_{\tau, \sigma_{I}} \mathbb{E}^{x_0}_{\tau,\sigma_{I},
  \sigma_{0, II}}[u\circ x_0 + \frac{\eta}{2^{0}}] = u(x_0) +\eta, 
\end{split}
\end{equation*}
where we used the definition of $u_{II}$, the fact that $G\leq u$, 
and the Doob's optional stopping theorem in view of the
supermartingale property and the uniform boundedness of the random
variables $\{u\circ x_{\tau\wedge n} + \frac{\eta}{2^{\tau\wedge
    n}}\}_{n\geq 0}$.
Since $\eta>0$ was arbitrary, (\ref{uII<u}) follows.

\smallskip

{\bf 2.}  We now prove that:
\begin{equation}\label{u<uI}
u\leq u_{I}  \qquad \mbox{in } \Omega.
\end{equation}
Together with (\ref{uII<u}) and in view of the direct observation from
(\ref{value}) that $u_{I}\leq u_{II}$, (\ref{u<uI}) will imply Theorem \ref{prob}.

Fix $\eta>0$ and fix any strategy $\sigma_{II}$.
By Lemma \ref{selection}, there exists a strategy $\sigma_{0, I}$ for
Player I, such that $\sigma_{0, I}^{n}(x_0, \ldots x_n) =
\sigma_{0, I}^{n}(x_n)$ and that:
\begin{equation}\label{almost2}
\forall n\geq 1\quad \forall x_n \in X \qquad
u(\sigma_{0, I}^{n}(x_n)) \geq \sup_{B_\epsilon(x_n)} u
- \frac{\eta}{2^{n+1}}.
\end{equation}
Define the stopping time $\bar\tau(\omega) = \inf\{n\geq 0; ~u(x_{n}) = \Psi(x_n)\},$
with the convention that $\inf$ over an empty set is $+\infty$.
Clearly:
\begin{equation}\label{raz} 
(x_0, \ldots x_{n})\not\in A_{n}^{\bar\tau}
\qquad \mbox{iff} \qquad \forall k=0\ldots n \quad u(x_k) > \Psi(x_k).
\end{equation}
As before, by Lemma \ref{condi}, the definition (\ref{gamma}), and the
suboptimality in (\ref{almost2}), we obtain:
\begin{equation*}
\begin{split}
 \forall (x_0, \ldots, x_{n-1})&\not\in A^{\bar\tau}_{n-1} \qquad
\mathbb{E}^{x_0}_{\tau,\bar\sigma_{0, I},
  \sigma_{II}}\{u\circ x_n - \frac{\eta}{2^n}\mid \mathcal{F}^{x_0}_{n-1}\} (x_0, \ldots,x_{n-1}) 
\\ & = \int_X u ~\mbox{d}\gamma_{n-1}[x_0,\ldots, x_{n-1}] - \frac{\eta}{2^n}\\
& = \frac{\alpha}{2} u(\sigma_{0, I}^{n-1}(x_0, \ldots, x_{n-1})) +
\frac{\alpha}{2} u(\sigma_{II}^{n-1}(x_0, \ldots, x_{n-1}))  
+ \beta\fint_{B_\epsilon(x_{n-1})} u - \frac{\eta}{2^n} \\
& \geq \frac{\alpha}{2} \sup_{B_\epsilon(x_{n-1})} u + \frac{\alpha}{2} 
\inf_{B_\epsilon(x_{n-1})} u + \beta\fint_{B_\epsilon(x_{n-1})} u -
\frac{\eta}{2^n}(1+\frac{\alpha}{2})\\
& = u(x_{n-1}) - \frac{\eta}{2^{n}}(1+\frac{\alpha}{2}) \geq \Big(u\circ x_{n-1} -
\frac{\eta}{2^{n-1}}\Big) (x_0, \ldots,x_{n-1}),
\end{split}
\end{equation*}
where the last equality above follows from (\ref{dpp}) because of (\ref{raz}).
For $(x_0, \ldots, x_{n-1})\in A^{\bar\tau}_{n-1}$, we also get, as
before: $\mathbb{E}^{x_0}_{\tau,\sigma_{0, I},
\sigma_{II}}\{u\circ x_n - \frac{\eta}{2^n}\mid \mathcal{F}^{x_0}_{n-1}\} (x_0, \ldots,x_{n-1}) 
=  u(x_{n-1}) - \frac{\eta}{2^{n}}$. We now conclude that $\{ u\circ
x_n - \frac{\eta}{2^n}\}_{n\geq 0}$ is a submartingale with respect
to the filtration $\{\mathcal{F}^{x_0}_{n}\}$, and therefore:
\begin{equation*}
\begin{split}
 u_{I}(x_0) &\geq \inf_{\sigma_{II}}
 \mathbb{E}^{x_0}_{\bar\tau,\sigma_{0, I},
  \sigma_{II}}[G\circ x_{\bar\tau} - \frac{\eta}{2^{\bar\tau}}] =
\inf_{\sigma_{II}} \mathbb{E}^{x_0}_{\bar\tau,\sigma_{0, I},
  \sigma_{II}}[u\circ x_{\bar\tau} - \frac{\eta}{2^{\bar\tau}}] \\ & \geq
\inf_{\sigma_{II}} \mathbb{E}^{x_0}_{\bar\tau,\sigma_{0, I},
  \sigma_{II}}[u\circ x_0 - \frac{\eta}{2^{0}}] = u(x_0) -\eta, 
\end{split}
\end{equation*}
where we used the definition of $u_{I}$, the fact that
$G(x_{\bar\tau}) = u(x_{\bar\tau})$ derived from the definition of $\bar\tau$, 
and the Doob's optional stopping theorem used to the two stopping
times: $\bar\tau$ and $0$. Since $\eta>0$ was arbitrary, we conclude (\ref{u<uI}).
\endproof

\section{The main convergence theorem: a proof of Theorem \ref{introthm}}\label{bo1}

First, we recall the definition of viscosity solutions.

\begin{definition}\label{viscosity}
We say that a continuous function $u:\bar\Omega\to\mathbb{R}$ is a
viscosity solution of the obstacle problem (\ref{obstacle}) if and
only if: $u=F$ on $\partial\Omega$ together with $u\geq \Psi$ in
$\Omega$, and:
\begin{itemize}
\item[(i)] for every $x_0\in\Omega$ and every
  $\phi\in\mathcal{C}^2(\Omega)$ such that:
\begin{equation}\label{ma1}
\phi(x_0) = u(x_0) , \quad \phi<u ~~\mbox{ in } ~\bar\Omega\setminus \{x_0\},
\quad \nabla \phi(x_0)\neq 0,
\end{equation}
there holds: $\Delta_p\phi(x_0)\leq 0$.
\item[(ii)] for every $x_0\in\Omega$ such that $u(x_0)>\Psi(x_0)$ and every
  $\phi\in\mathcal{C}^2(\Omega)$ such that:
\begin{equation}\label{ma2}
\phi(x_0) = u(x_0) , \quad \phi>u ~~\mbox{ in } ~\bar\Omega\setminus \{x_0\},
\quad \nabla \phi(x_0)\neq 0,
\end{equation}
there holds: $\Delta_p\phi(x_0)\geq 0$.
\end{itemize}
\end{definition}

The fact that variational solutions are viscosity solutions in the
sense of Definition \ref{viscosity} is due to the equivalence of the
local notions of viscosity and weak solutions \cite{JLM}  and the
continuity up to the boundary of variational solutions under the
regularity hypothesis on $F$, $\Psi$, and $\Omega$ (see for example
\cite{BB}.)  The fact that viscosity solutions are variational
solutions is actually equivalent to the following folklore uniqueness result that, for the sake of
completeness, we prove in the Appendix:

\begin{lemma}\label{uni}
Let $u$ and $\bar u$ be two viscosity solutions to (\ref{obstacle}) as
in Definition \ref{viscosity}. Then $u=\bar u$. 
\end{lemma}
It is also classical that the unique solution to (\ref{obstacle}) is  the pointwise infimum
of all $p$-superharmonic functions that are above the obstacle (see
Chapters 5 and 7 in \cite{HKM}). 

\medskip

Our main approximation result is given in Theorem \ref{approx} below.

\begin{theorem}\label{approx}
Let $p\in [2,\infty)$. Let $F:\partial\Omega\rightarrow \mathbb{R}$,
$\Psi:\bar\Omega\to\mathbb{R}$  be two Lipschitz continuous
functions, satisfying:
\begin{equation}\label{dwa2}
\Psi\leq F \qquad \mbox{on } \partial\Omega.
\end{equation}
Without loss of generality, we may assume that $F, \Psi$ above are defined on
$\bar\Gamma$ and $X$, respectively, and
that (\ref{dwa2}) still holds on $\Gamma$.
Let $u_\epsilon:\Omega\cup\Gamma\to\mathbb{R}$ be the unique
$\epsilon$-$p$-superharmonious function solving (\ref{dpp}) with $\alpha =
\frac{p-2}{p+N}$ and $\beta = \frac{2+N}{p+N}$.

Then $u_\epsilon$ converge as $\epsilon \to 0$,
uniformly in $\bar\Omega$, to a continuous function $u$ which is the
unique viscosity solution to the obstacle problem (\ref{obstacle}).
\end{theorem}

\begin{proof}
{\bf 1.} We first prove the uniform convergence of $u_\epsilon$, as
$\epsilon\to 0$, in $\bar\Omega$. This is achieved by verifying the
assumptions of the following version of the Ascoli-Arzel\'a theorem,
valid for equibounded (possibly discontinuous) functions with ``uniformly vanishing oscillation'':

\begin{lemma}\label{AA}\cite{MPR} 
Let $u_\epsilon:\bar\Omega\to\mathbb{R}$ be a set of functions such
that:

(i) $~\exists C>0\quad \forall \epsilon >0
\qquad \|u_\epsilon\|_{L^\infty(\bar\Omega)} \leq C,$

(ii) $~\forall \eta>0\quad \exists r_0, \epsilon_0>0\quad \forall
\epsilon<\epsilon_0 \quad \forall x_0,y_0\in\bar\Omega \qquad |x_0-y_0|<r_0
\implies |u_\epsilon(x_0) - u_\epsilon(y_0)|<\eta$

\noindent Then, a subsequence of $u_\epsilon$ converges uniformly in
$\bar\Omega$, to a continuous function $u$.
\end{lemma}

Clearly, solutions $u_\epsilon$ to (\ref{dpp}) as in the statement of
Theorem \ref{approx} are uniformly bounded,
by the boundedness of $F$.  The crucial step in the proof of 
condition (ii) above is achieved by estimating the oscillation of
$u_\epsilon$ close to the boundary.

\begin{lemma}\label{closetoboundary}
Under the assumptions of Theorem \ref{approx},
let $u_\epsilon:X\to\mathbb{R}$ be the $\epsilon$-$p$-superharmonious
solution to (\ref{dpp}). Then, for every $\eta>0$ there exist $r_0,
\epsilon_0>0$ such that we have:
\begin{equation}\label{bobo}
\forall \epsilon<\epsilon_0 \quad \forall y_0\in\partial\Omega \quad 
\forall x_0\in\bar\Omega \quad |x_0 - y_0|< r_0 \implies
|u_\epsilon(x_0) - u_\epsilon (y_0)|<\eta.
\end{equation}
\end{lemma}

We postpone the proof of Lemma \ref{closetoboundary} to Section \S
\ref{bo2}.  We now have:

\begin{corollary}\label{ascoli}
Under the assumptions of Theorem \ref{approx},
let $\{u_\epsilon\}$ be the sequence of $\epsilon$-$p$-superharmonious
solutions to (\ref{dpp}). Then $\{u_\epsilon\}$ satisfies condition (ii) in Lemma \ref{AA}.
\end{corollary}
\begin{proof}
Fix $\eta>0$ and let $r_0, \epsilon_0$ be as in Lemma
\ref{closetoboundary} so that (\ref{bobo}) holds with
$\eta/3$ instead of $\eta$. Since $\Psi$ is Lipschitz, we may without
loss of generality also assume that:
\begin{equation}\label{comppsi}
\forall x,y\in X \quad |x-y|< r_0 \implies |\Psi(x) - \Psi(y)|<\eta.
\end{equation}
Note that, consequently, we have:
\begin{equation}\label{bobo2}
\forall \epsilon<\epsilon_0\quad \forall x_0, y_0\in
\tilde\Gamma_{r_0/3}\quad |x_0-y_0|<r_0/3 \implies |u_\epsilon(x_0) -
u_\epsilon(y_0)|\leq \eta,
\end{equation}
where for any $\delta>0$ we denote:
$$\tilde\Gamma_\delta=\{x\in\bar\Omega; ~ \mbox{dist}(x, \partial\Omega)\leq \delta\}.$$
In particular, the same implication as in (\ref{bobo2}) holds for
$x_0\in\tilde\Gamma_{r_0/6}$ and $y_0\in\bar\Omega$ when $|x_0-
y_0|<r_0/6$.

Let now $x_0, y_0\in\Omega\setminus\tilde\Gamma_{r_0/6}$ and assume that
$|x_0-y_0|<r_0/6$. Define the bounded Borel function $\tilde F:\tilde
\Gamma_{r_0/6}\to\mathbb{R}$ and the Lipschitz obstacle
$\tilde\Psi:\mathbb{R}^N\to\mathbb{R}$ by:
$$\tilde F(z) = u_\epsilon(z-(x_0-y_0)) + \eta, \qquad \tilde \Psi(z) = \Psi(z-(x_0-y_0))+\eta.$$
Clearly: $\tilde F\geq\tilde\Psi$ in $\tilde\Gamma_{r_0/6}$, hence by
Theorem \ref{exists} there exists a solution $\tilde
u_\epsilon:\Omega\to\mathbb{R}$ to (\ref{dpp}) subject to the boundary data
$\tilde F$ on $\tilde\Gamma_{r_0/6}$, and to the obstacle constraint
$\tilde\Psi$. Note that by the uniqueness of such solution, there must be:
$$\tilde u_\epsilon(z) = u_\epsilon(z-(x_0-y_0))+\eta$$
On the other hand: $\tilde F\geq u_\epsilon$ in $\tilde\Gamma_{r_0/6}$
by (\ref{bobo2}), and also: $\tilde\Psi\geq \Psi$ in $\Omega$ by
(\ref{comppsi}). Corollary \ref{comparison} now implies that $\tilde
u_\epsilon\geq u_\epsilon$ in $\bar\Omega$ and we get:
$$u_\epsilon(x_0) -u_\epsilon(y_0) \leq \tilde u_\epsilon(x_0) -
u_\epsilon(y_0) = u_\epsilon (y_0) +\eta - u_\epsilon(y_0) =\eta.$$
Exchanging $x_0$ with $y_0$, the same argument yields 
$|u_\epsilon(x_0) -u_\epsilon(y_0)|<\eta$. 
\end{proof}

\medskip

{\bf 2.} We now prove that the uniform limit of $u_\epsilon$ is a
viscosity solution to the obstacle problem (\ref{obstacle}). Clearly,
$u=F$ on $\partial\Omega$ and $u\geq \Psi$ in $\Omega$ because each
$\epsilon$-$p$-superharmonious function $u_\epsilon$ has the same
properties. We show that (i) in Definition \ref{viscosity} holds.

Let $\phi$ be a test function as in (\ref{ma1}). Since $x_0$ is the
minimum of the continuous function $u-\phi$, one can find a sequence
of points $x_\epsilon$ converging to $x_0$ as $\epsilon\to 0$, and
such that:
\begin{equation}\label{apro}
u_\epsilon(x_\epsilon) - \phi(x_\epsilon)\leq \inf_{\bar\Omega}~
(u_\epsilon - \phi) + \epsilon^3.
\end{equation}
To prove this statement, for every $j\geq 1$ let $a_j =
\displaystyle{\min_{\bar\Omega\setminus B_{{1}/{j}}(x_0)} (u-\phi)} >0$ and let
$\epsilon_j>0$ be such that: 
$$\forall \epsilon\leq \epsilon_j \qquad \|u_\epsilon -
u\|_{L^\infty(\Omega)} \leq\frac{1}{2}a_j.$$
Without loss of generality $\{\epsilon_j\}$ is decreasing
to $0$. Now, for $\epsilon\in (\epsilon_{j+1},\epsilon_j]$ let
$x_\epsilon\in B_{1/j}(x_0)$ satisfy:
$$u_\epsilon(x_\epsilon) - \phi(x_\epsilon) \leq \inf_{B_{1/j}(x_0)}~
(u_\epsilon - \phi) + \epsilon^3.$$
We finally conclude (\ref{apro}) by noting that also for every $x\in \bar\Omega\setminus B_{1/j}(x_0)$ there
holds:
\begin{equation*}
\begin{split}
u_\epsilon(x) - \phi(x) & \geq u(x) - \phi(x) - \|u_\epsilon -
u\|_{L^\infty(\Omega)} \geq a_j - \frac{1}{2} a_j \geq \|u_\epsilon -
u\|_{L^\infty(\Omega)}\\ & \geq u_\epsilon(x_0) - u(x_0) = u_\epsilon(x_0)
- \phi(x_0) \geq u_\epsilon(x_\epsilon) - \phi(x_\epsilon) - \epsilon^3.
\end{split}
\end{equation*}

By (\ref{apro}) it follows that for all $x\in\Omega$ we have:
$u_\epsilon(x) \geq u_\epsilon(x_\epsilon) - \phi(x_\epsilon) +
\phi(x) -\epsilon^3$ and hence:
\begin{equation}\label{mama}
\begin{split}
u_\epsilon(x_\epsilon) & \geq \frac{\alpha}{2} \sup_{B_\epsilon(x_\epsilon)} u_\epsilon + \frac{\alpha}{2}
\sup_{B_\epsilon(x_\epsilon)} u_\epsilon
+\beta\fint_{B_\epsilon(x_\epsilon)} u_\epsilon \\ & \geq
\big(u_\epsilon(x_\epsilon) - \phi(x_\epsilon)  - \epsilon^3 \big) +  \Big(
\frac{\alpha}{2} \sup_{B_\epsilon(x_\epsilon)} \phi + \frac{\alpha}{2}\sup_{B_\epsilon(x_\epsilon)}\phi
+\beta\fint_{B_\epsilon(x_\epsilon)} \phi\Big),
\end{split}
\end{equation}
which further implies, for $\bar x_\epsilon\in\mbox{argmin}_{\bar B_\epsilon(x_\epsilon)}\phi$:
\begin{equation}\label{man}
\begin{split}
\epsilon^3 & \geq \Big(
\frac{\alpha}{2} \sup_{B_\epsilon(x_\epsilon)} \phi + \frac{\alpha}{2}\sup_{B_\epsilon(x_\epsilon)}\phi
+\beta\fint_{B_\epsilon(x_\epsilon)} \phi\Big) - \phi(x_\epsilon) \\ &
\geq \frac{\beta\epsilon^2}{2(N+2)} \Big((p-2) \Big\langle
\nabla^2\phi(x_\epsilon) \frac{(\bar x_\epsilon -
  x_\epsilon)}{\epsilon}, \frac{\bar x_\epsilon -
  x_\epsilon}{\epsilon}\Big\rangle + \Delta\phi(x_\epsilon)\Big) +o(\epsilon^2).
\end{split}
\end{equation}

For completeness, we recall now \cite{MPR0} the proof of the second inequality in (\ref{man}).
Taylor expand the regular function $\phi$ at $x_\epsilon$, to get:
\begin{equation*}
\min_{\bar B_\epsilon(x_\epsilon)}\phi = \phi(\bar x_\epsilon) =
\phi(x_\epsilon) + \langle\nabla \phi(x_\epsilon), \bar x_\epsilon -
x_\epsilon\rangle + \frac{1}{2} \big\langle\nabla^2 \phi(x_\epsilon)(\bar x_\epsilon -
x_\epsilon), (\bar x_\epsilon - x_\epsilon)\big\rangle + o(\epsilon^2).
\end{equation*}
On the other hand, in a similar manner:
\begin{equation*}
\max_{\bar B_\epsilon(x_\epsilon)}\phi \geq \phi(x_\epsilon
+(x_\epsilon - \bar x_\epsilon)) =
\phi(x_\epsilon) - \langle\nabla \phi(x_\epsilon), \bar x_\epsilon -
x_\epsilon\rangle + \frac{1}{2} \big\langle\nabla^2 \phi(x_\epsilon)(\bar x_\epsilon -
x_\epsilon), (\bar x_\epsilon - x_\epsilon)\big\rangle + o(\epsilon^2),
\end{equation*}
and again:
\begin{equation*}
\fint_{B_\epsilon(x_\epsilon)}\phi = \phi(\bar x_\epsilon) +
\frac{\epsilon^2}{2(N+2)} \Delta \phi(x_\epsilon) + o(\epsilon^2).
\end{equation*}
Consequently, we obtain:
\begin{equation*}
\begin{split}
&\Big(\frac{\alpha}{2} \max_{\bar B_\epsilon(x_\epsilon)} \phi +
\frac{\alpha}{2}\min_{\bar B_\epsilon(x_\epsilon)}\phi
+\beta\fint_{B_\epsilon(x_\epsilon)} \phi\Big) - \phi(x_\epsilon) \\ &
\qquad\qquad\qquad\qquad
\geq \frac{\alpha}{2} \Big\langle\nabla^2 \phi(x_\epsilon)(\bar x_\epsilon -
x_\epsilon), (\bar x_\epsilon - x_\epsilon)\Big\rangle +
\frac{\beta\epsilon^2}{2(N+2)} \Delta \phi(x_\epsilon) +
o(\epsilon^2),
\end{split}
\end{equation*}
which yields (\ref{man}), because $(\frac{\alpha}{2}) / (\frac{\beta\epsilon^2}{2(N+2)}) = \frac{p-2}{\epsilon^2}$.

\medskip

{\bf 3.} After dividing by $\epsilon^2$, (\ref{man}) becomes the
following asymptotic inequality, valid for $\epsilon\to 0$:
\begin{equation}\label{man2}
\begin{split}
\limsup_{\epsilon\to 0}\Big((p-2) \Big\langle
\nabla^2\phi(x_\epsilon) \frac{(\bar x_\epsilon -
  x_\epsilon)}{\epsilon}, \frac{\bar x_\epsilon -
  x_\epsilon}{\epsilon}\Big\rangle + \Delta\phi(x_\epsilon)\Big) \leq 0.
\end{split}
\end{equation}
Note now that:
\begin{equation}\label{man3}
\lim_{\epsilon\to 0} \frac{(\bar x_\epsilon -  x_\epsilon)}{\epsilon}
= -\frac{\nabla \phi(x_0)}{|\nabla \phi(x_0)|}.
\end{equation}
This follows by a simple blow-up argument. Indeed, the maps
$\phi_\epsilon(z) = \frac{1}{\epsilon} \big(\phi(x_\epsilon + \epsilon
z) - \phi(z_\epsilon)\big)$ converge uniformly on $\bar B_1(0)$ to the
linear map $\langle \nabla \phi(x_0),z\rangle$. Hence the limit of any converging
subsequence of their minimizers: $\frac{1}{\epsilon} (\bar x_\epsilon - x_\epsilon) \in
\mbox{argmin}_{\bar B_1(0)}\phi_\epsilon $ must be a minimizer of the
limiting function $\langle \nabla \phi(x_0),z\rangle$. This minimizer
is unique and equals: $-\nabla \phi(x_0) /|\nabla \phi(x_0)|$, proving (\ref{man3}).

In conclusion, (\ref{man3}) and (\ref{man2}) imply that:
\begin{equation*}
\begin{split}
\frac{1}{|\nabla \phi|^{p-2}}\Delta_p \phi(x_0) = 
(p-2) \Big\langle \nabla^2\phi(x_0) \frac{\nabla \phi(x_0)}{|\nabla
  \phi(x_0)|}, \frac{\nabla \phi(x_0)}{|\nabla \phi(x_0)|} \Big\rangle + \Delta\phi(x_0)\leq  0,
\end{split}
\end{equation*}
which yields the validity of condition  (i) in Definition \ref{viscosity}, in view of  (\ref{plapla}).

\medskip

{\bf 4.} To prove condition (ii) in Definition \ref{viscosity}, let
$\phi$ be as in (\ref{ma2}). One can follow the argument as in steps
2. and 3. above, taking $x_\epsilon$ to be the approximate maximizers
of $u_\epsilon-\phi$. The first inequality in (\ref{mama}) is then
replaced by equality because $u(x_0)>\Psi(x_0)$, and we consequently obtain:
\begin{equation*}
\begin{split}
u_\epsilon(x_\epsilon) & \leq 
\big(u_\epsilon(x_\epsilon) - \phi(x_\epsilon)  + \epsilon^3 \big) +  \Big(
\frac{\alpha}{2} \sup_{B_\epsilon(x_\epsilon)} \phi + \frac{\alpha}{2}\sup_{B_\epsilon(x_\epsilon)}\phi
+\beta\fint_{B_\epsilon(x_\epsilon)} \phi\Big),
\end{split}
\end{equation*}
while the counterpart of (\ref{man}), written for $\bar
x_\epsilon\in\mbox{argmax}_{\bar B_\epsilon(x_\epsilon)}\phi$ is:
\begin{equation*}
\begin{split}
-\epsilon^3  \leq \frac{\beta\epsilon^2}{2(N+2)} \Big((p-2) \Big\langle
\nabla^2\phi(x_\epsilon) \frac{(\bar x_\epsilon -
  x_\epsilon)}{\epsilon}, \frac{\bar x_\epsilon -
  x_\epsilon}{\epsilon}\Big\rangle + \Delta\phi(x_\epsilon)\Big) +o(\epsilon^2).
\end{split}
\end{equation*}
Similarly to step 3. above, we conclude: $\Delta_p\phi(x_0)\geq 0$.
The proof of Theorem \ref{approx} is complete.
\end{proof}

\section{Estimates close to the boundary: a proof of Lemma \ref{closetoboundary}}\label{bo2}

In this Section, by $C$ we denote constants that depend only on the
general setup of the problem, i.e. on $N$, $\Omega$, $p$, $\alpha$ and
$\beta$, but not on $u$, $x_0$, $F$ or $\Psi$. By $C_F, C_\Psi$ or $C_{F, \Psi}$
we denote constants depending additionally on $F$, $\Psi$, or on both $F$
and $\Psi$. 

Let $x_0\in\Omega$ and $y_0\in\partial\Omega$. Assume that we have
fixed a particular strategy $\sigma_{0, II}$  of Player II. Then, by (\ref{value}):
$$u_\epsilon(x_0) - u_\epsilon(y_0) \leq \sup_{\tau, \sigma_I}~
\mathbb{E}^{x_0}_{\tau,\sigma_I, \sigma_{0, II}} [G\circ x_\tau - F(y_0)].$$
Note that for every $x\in X$:
$$G(x) - F(y_0) \leq \chi_\Gamma(x) \big(F(x) - F(y_0)\big) +
\chi_\Omega (x) \big(\Psi(x) - \Psi(y_0)\big) \leq C_{F,\Psi} |x-y_0|,$$
thus:
\begin{equation}\label{one1}
u_\epsilon(x_0) - u_\epsilon(y_0) \leq C_{F,\Psi} \sup_{\tau, \sigma_I}~
\mathbb{E}^{x_0}_{\tau,\sigma_I, \sigma_{0, II}} [|x_\tau - y_0|].
\end{equation}
On the other hand, for a fixed strategy $\sigma_{0, I}$, again in view
of (\ref{value}) it follows that:
\begin{equation}\label{two1}
\begin{split}
u_\epsilon(x_0) - u_\epsilon(y_0) & \geq \inf_{\sigma_{II}}~
\mathbb{E}^{x_0}_{\tau_0,\sigma_{0, I}, \sigma_{II}} [G\circ x_{\tau_0}
- F(y_0)] = \inf_{\sigma_{II}}~
\mathbb{E}^{x_0}_{\tau_0,\sigma_{0, I}, \sigma_{II}} [F\circ x_{\tau_0}
- F(y_0)] \\ & \geq -C_F \sup_{\sigma_{II}}~
\mathbb{E}^{x_0}_{\tau_0,\sigma_{0, I}, \sigma_{II}} [|x_\tau - y_0|].
\end{split}
\end{equation}
We will now prove that, with $\sigma_{0, I}$ and $\sigma_{0, II}$
chosen appropriately, one has:
\begin{equation}\label{three}
\begin{split}
\forall 0<\delta\ll 1& \quad  \forall \epsilon<\min
\big(\frac{\beta}{2C_\delta}, \frac{\delta}{3}\big) \quad \\ & \sup_{\tau, \sigma_{I}}~
\mathbb{E}^{x_0}_{\tau,\sigma_I, \sigma_{0, II}} [|x_\tau - y_0|] +
\sup_{\tau, \sigma_{II}}~ \mathbb{E}^{x_0}_{\tau,\sigma_{0, I}, \sigma_{II}} [|x_\tau - y_0|]
\leq C \delta + C_\delta (|x_0 - y_0| + \epsilon), 
\end{split}
\end{equation}
where the supremum is taken over all admissible stopping times
$\tau\leq \tau_0$. The constant $C_\delta$ depends only on the
associated parameter $\delta$ in (\ref{three}). Clearly, (\ref{three})
with (\ref{one1}) and (\ref{two1}) will imply (\ref{bobo}).

\begin{remark}
Denote by $u_\epsilon^0$ the $\epsilon$-$p$-superharmonious function subject to the
same boundary condition $F$ as $u_\epsilon$ on $\Gamma$, but in the absence
of any obstacle. It satisfies:
$$u_\epsilon^0(x_0) = \sup_{\sigma_I} \inf_{\sigma_{II}}
\mathbb{E}^{x_0}_{\tau_0, \sigma_I, \sigma_{II}} [F\circ x_{\tau_0}] 
= \inf_{\sigma_{II}}  \sup_{\sigma_I}
\mathbb{E}^{x_0}_{\tau_0, \sigma_I, \sigma_{II}} [F\circ
x_{\tau_0}].$$ 
Equivalently, $u_\epsilon^0$ solves (\ref{dpp}) with $\Psi=\mbox{const}
<\min_\Gamma F$. It is clear that:
\begin{equation}\label{com}
u_\epsilon\geq u_\epsilon^0 \quad \mbox{ in } \bar\Omega.
\end{equation}
The following estimate has been proven in \cite{MPR}:
\begin{equation}\label{u0e}
\forall y_0\in\partial\Omega, ~ x_0\in\Omega\quad \forall 0<\delta\ll
1\quad |u_\epsilon^0(x_0) - u_\epsilon^0(y_0)|\leq C_F \delta
+ C_\delta (|x_0 - y_0| + \epsilon).
\end{equation}
Note that the lower bound follows directly by (\ref{com}) and (\ref{u0e}):
$$u_\epsilon(x_0) - u_\epsilon(y_0) \geq u_\epsilon^0(x_0) -
u_\epsilon^0(y_0) \geq -\Big( C_F \delta + C_\delta (|x_0 - y_0| + \epsilon)\Big).$$
Also, the upper bound is straightforward in case when $x_0$ belongs to the contact
set, i.e. when: $u_\epsilon(x_0) = \Psi(x_0)$, because then in view of (\ref{dwa2}):
$$u_\epsilon(x_0) - u_\epsilon(y_0)  = \Psi(x_0) - F(y_0) \leq
\Psi(x_0) - \Psi(y_0)\leq C_\Psi |x_0 - y_0|.$$
It remains hence to prove a similar bound for the case
$x_0\in\Omega\setminus A_\epsilon$.
We will in fact reprove the inequality (\ref{u0e}), in a slightly more
general setting of the obstacle $\epsilon$-$p$-superharmonious function
$u_\epsilon$. The scheme of proof of (\ref{three}) below follows \cite{MPR}
but we fill in all the details.
\endproof
\end{remark}

\bigskip

\noindent{\bf Proof of Lemma \ref{closetoboundary}.}

{\bf 1.} Let $\delta>0$ and  $z_0\in\mathbb{R}^N\setminus\Omega$ satisfy:
$B_\delta(z_0)\cap\bar\Omega = \{y_0\}.$ Define strategy
$\sigma_{0, II}$ for Player II:
\begin{equation}\label{sn}
\sigma_{0, II}^n(x_0,\ldots x_n) = \sigma_{0, II}^n(x_n) = \left\{\begin{array}{ll} x_n
    +(\epsilon-\epsilon^3)\frac{z_0-x_{n}}{|z_0 - x_n|} & \mbox{ if }
    x_n\in\Omega\\ x_n & \mbox{ if } x_n\in\Gamma.
\end{array}\right.
\end{equation}
Let $\sigma_I$ be any strategy for Player I and let $\tau\leq\tau_0$
be any admissible stopping time. 

Firstly, notice that for all $\epsilon<\delta/3$ we have:
\begin{equation}\label{aiuto}
\forall x\in\Omega \qquad \fint_{B_\epsilon(x)} |w-z_0|~\mbox{d}w
\leq |x-z_0| + C_\delta \epsilon^2.
\end{equation}
This is because the function $f(w) = |w-z_0|$ is smooth in the domain
$\Omega +B_{\delta/2}(0)$ and hence by taking Taylor's expansion and
averaging, we get: $\fint_{B_\epsilon(x)} f = u(x) +
\frac{\epsilon^2}{2(N+2)} \Delta f(x) + o(\epsilon^2)$.

Take $C=C_\delta +1$. 
By Lemma \ref{condi}, the definition (\ref{sn}), and (\ref{aiuto}) it follows that:
\begin{equation*}
\begin{split}
\forall (x_0,&\ldots x_{n-1})\not\in A_{n-1}^\tau \qquad 
\mathbb{E}^{x_0}_{\tau, \sigma_I, \sigma_{0, II}} \left\{|x_n -z_0|
  -  C\epsilon^2n | \mathcal{F}_{n-1}^{x_0}\right\}(x_0,\ldots
x_{n-1}) \\ & \leq \frac{\alpha}{2} |\sigma_I^{n-1}(x_0,\ldots x_{n-1}) - z_0| +
\frac{\alpha}{2} |\sigma_{0, II}^{n-1}(x_{n-1}) - z_0|  + \beta\fint_{B_\epsilon(x_{n-1})}
|w-z_0|~\mbox{d}w - C\epsilon^2 n \\ & \leq
\frac{\alpha}{2} (|x_{n-1}-z_0| + \epsilon) + 
\frac{\alpha}{2} (|x_{n-1}-z_0| - (\epsilon -\epsilon^3)) 
+ \beta (|x_{n-1} - z_0| + C_\delta \epsilon^2) - C\epsilon^2 n \\ & \leq |x_{n-1}- z_0| - C\epsilon^2(n-1),
\end{split}
\end{equation*}
while for $ (x_0,\ldots x_{n-1})\in A_{n-1}^\tau$ one has: 
$\mathbb{E}^{x_0}_{\tau, \sigma_I, \sigma_{0, II}} \left\{|x_n -z_0| 
  -  C\epsilon^2n | \mathcal{F}_{n-1}^{x_0}\right\}(x_0,\ldots
x_{n-1})  = |x_{n-1} -z_0| -  C\epsilon^2n \leq  |x_{n-1} -z_0| -
C\epsilon^2(n-1)$.
In any case, we see that $\{|x_n -z_0| -  C\epsilon^2n\}_{n\geq 0}$ is
a supermartingale with respect to the filtration $\{\mathcal{F}^{x_0}_n\}$. Applying Doob's
theorem to the bounded stopping times $\tau\wedge n$ and $0$, we obtain:
\begin{equation*}
\mathbb{E}^{x_0}_{\tau,\sigma_I,\sigma_{0, II}}[|x_{\tau\wedge n} -
z_0|] - C\epsilon^2 \mathbb{E}^{x_0}_{\tau,\sigma_I,\sigma_{0, II}}
[\tau\wedge n] \leq |x_0 - z_0|.
\end{equation*}
Consequently, and further using
the dominated and the monotone convergence theorems while passing with
$n\to\infty$ and recalling (\ref{koniec}), we obtain:
\begin{equation}\label{pies}
\mathbb{E}^{x_0}_{\tau,\sigma_I,\sigma_{0, II}}[|x_\tau - y_0|] \leq
\mathbb{E}^{x_0}_{\tau,\sigma_I,\sigma_{0, II}}[|x_\tau -
z_0|] + \delta \leq |x_0 - y_0| + 2\delta + C_\delta\epsilon^2
\mathbb{E}^{x_0}_{\tau,\sigma_I,\sigma_{0, II}} [\tau]. 
\end{equation}

\medskip

{\bf 2.} Towards estimating the expectation
$\mathbb{E}^{x_0}_{\tau,\sigma_I,\sigma_{0, II}} [\tau]$ in (\ref{pies}),
we first observe the following simple general result:

\begin{lemma}\label{transitionlemma}
Let $Y\subset\mathbb{R}^N$ be a bounded, open set. Let $x_0\in Y$ and let
$\{\mathbb{P}^{n, x_0}\}_{n\geq 1}$, $\{\bar{\mathbb{P}}^{n,
  x_0}\}_{n\geq 1}$  be two consistent sequences of probability
measures on $Y^{\infty, x_0}$ defined as in (\ref{partial}), where
the filtration $\{\mathcal{F}^{x_0}_n\}_{n\geq 1}$ is as in
subsection \S \ref{start}, and the transition probabilities are denoted,
respectively: $\{\gamma_n[x_0, x_1,\ldots x_n]\}$ and $\{\bar\gamma_n[x_0,
x_1,\ldots x_n]\}$. Let $A\in\mathcal{F}^{n, x_0}$ have the property:
\begin{equation}\label{assu}
\forall \omega=(x_0, x_1,\ldots )\in A \quad \forall 0\leq k < n\qquad
\gamma_k[x_0, x_1,\ldots x_k] = \bar \gamma_k[x_0, x_1,\ldots x_k].
\end{equation}
Then $\mathbb{P}^{n,x_0}(A) = \bar{\mathbb{P}}^{n,x_0}(A)$.
\end{lemma}
\begin{proof}
We prove the lemma by induction on $n$. For $n=1$, the result is
trivially true because $\mathbb{P}^{1, x_0} =  \bar{\mathbb{P}}^{1, x_0}
=\gamma_0[x_0]$.

Let now $A\subset \{x_0\}\times Y^{n}$ be a Borel set. Fix
$\eta>0$ and find the covering $A\subset\bigcup_{i=1}^\infty (A_1^i\times A_2^i)$
where each $A_1^i\subset \{x_0\}\times Y^{n-1}$ and $A_2^i\subset Y$
is a Borel set, such that the rectangles $\{A_1^i\times A_2^i\}$ are pairwise
disjoint, and such that:
\begin{equation}\label{assu1}
0\leq \Big(\sum_{i=1}^\infty \mathbb{P}^{n, x_0} (A_1^i\times A_2^i) -
\mathbb{P}^{n, x_0}(A) \Big) + \Big(\sum_{i=1}^\infty
\bar{\mathbb{P}}^{n,  x_0} (A_1^i\times A_2^i) - \bar{\mathbb{P}}^{n, x_0}(A) \Big) \leq\eta.
\end{equation}
For each $i\geq 1$ consider the Borel set $A^i = A\cap (A^i_1 \times
A^i_2)$. Its projection:
$$\pi(A^i) = \{(x_0, x_1,\ldots x_{n-1}):~ \exists x_n\in Y \quad (x_0,\ldots x_n)\in A^i\}$$
does not have to be Borel (compare Remark \ref{remi}), but it is an analytic set \cite{RS} and
hence it is measurable with respect to completions of Borel
measures. Hence, there exists Borel sets $B_1^i, C_1^i\subset
\{x_0\}\times Y^{n-1}$ such that:
$$B_1^i\subset\pi(A^i)\subset C_1^i \quad \mbox{and} \quad
\mathbb{P}^{n-1, x_0}(C_1^i\setminus B_1^i) = \bar{\mathbb{P}}^{n-1,
  x_0}(C_1^i\setminus B_1^i) = 0.$$
We have then:
\begin{equation}\label{assu2}
\begin{split}
\mathbb{P}^{n, x_0}(B_1^i\times A_2^i) & = \int_{B_1^i}\gamma_{n-1}[x_0,\ldots x_{n-1}]
(A_2^i)~\mbox{d}\mathbb{P}^{n-1, x_0} \\ & = 
\int_{B_1^i}\bar\gamma_{n-1}[x_0,\ldots x_{n-1}] (A_2^i)~\mbox{d}\bar{\mathbb{P}}^{n-1, x_0} = 
\bar{\mathbb{P}}^{n, x_0}(B_1^i\times A_2^i), 
\end{split}
\end{equation}
where we used the induction assumption to conclude that
$\mathbb{P}^{n-1, x_0}\lfloor B_1^i = \bar{\mathbb P}^{n-1,
  x_0}\lfloor B_1^i$. 

Further, in view of  (\ref{assu1}) and the fact that: 
$$\mathbb{P}^{n,
  x_0}((C_1^i\setminus B_1^i) \times A_2^i) = \int_{C_1^i\setminus
  B_1^i}  \gamma_{n-1}[x_0,\ldots x_{n-1}](A_2^i)~\mbox{d}\mathbb{P}^{n-1,
  x_0}= 0,$$ 
there holds:
\begin{equation*}
\begin{split}
|\mathbb{P}^{n, x_0}(A) - \sum_{i=1}^\infty \mathbb{P}^{n,
  x_0} (B_1^i\times A_2^i) | & = |\sum_{i=1}^n\mathbb{P}^{n, x_0}(A^i) - \sum_{i=1}^\infty \mathbb{P}^{n,
  x_0} (C_1^i\times A_2^i) | \\ & \leq |\sum_{i=1}^n\mathbb{P}^{n, x_0}(A^i) - \sum_{i=1}^\infty \mathbb{P}^{n,
  x_0} (A_1^i\times A_2^i) | \leq \eta.
\end{split}
\end{equation*}
In a similar manner
$ |\bar{\mathbb{P}}^{n, x_0}(A) - \sum_{i=1}^\infty \bar{\mathbb{P}}^{n,
  x_0} (B_1^i\times A_2^i) | \leq \eta.$
In view of (\ref{assu2}), this implies: $ |\bar{\mathbb{P}}^{n,
  x_0}(A) - \mathbb{P}^{n, x_0}(A) |\leq 2\eta. $ Since $\eta>0$ was
arbitrary, the lemma follows.
\end{proof}

\medskip

Consider now a new ``game-board'' $Y=B_R(z_0)\supset X$ with
the same initial token position $x_0$. Let $\bar\sigma_I$ be an
extension of the strategy $\sigma_I$, given by:
\begin{equation}\label{easy_ext}
\forall (x_0, \ldots x_n)\in Y^{n+1}\qquad \bar\sigma^n_I (x_0, \ldots x_n) = \left\{\begin{array}{ll}
    \sigma_I^n(x_0,\ldots x_n) & \mbox{ if } (x_0, \ldots x_n)\in  X^{n+1}\\
x_n & \mbox{otherwise,}  \end{array}\right. 
\end{equation}
while:
\begin{equation*}
\bar\sigma_{0, II}^n(x_0,\ldots x_n) = \bar\sigma_{0, II}^n(x_n) =
\left\{\begin{array}{ll} x_n
    +(\epsilon-\epsilon^3)\frac{z_0-x_{n}}{|z_0 - x_n|} & \mbox{ if }
    x_n\in Y\setminus \bar B_\delta(z_0) \\ x_n & \mbox{ otherwise. }
\end{array}\right.
\end{equation*}
Let $\bar\tau_0: Y^{\infty, x_0}\to\mathbb{N}\cup \{+\infty\}$ be the
exit time into the ball $\bar B_\delta(z_0)$, i.e.:
$$ \bar\tau_0(\omega) = \min\{n\geq 0; ~  |x_n-z_0|\leq \delta\}$$
and let $\bar\tau: Y^{\infty, x_0} \to\mathbb{N}\cup \{+\infty\}$ be
a stopping time extending $\tau$, so that $\bar\tau\leq
\bar\tau_0$ and $\bar\tau_{\mid X^{\infty, x_0}} = \tau$.
Define the transition probabilities on $Y$ by:
\begin{equation*}
\begin{split}
& \forall n\geq 1 \quad \forall x_1, \ldots, x_n\in Y\qquad  
\bar\gamma_n[x_0, x_1, \ldots, x_n] = \\ & \qquad = \left\{ \begin{array}{ll}
\displaystyle{\frac{\alpha}{2} \delta_{\bar\sigma_I^n(x_0,, \ldots, x_n)} +
\frac{\alpha}{2} \delta_{\bar\sigma_{0, II}^n(x_n)} + 
\beta m(x_n)} & \mbox{ for } x_n\in Y\setminus \bar B_\delta(z_0) \vspace{2mm}\\ 
\alpha\delta_{x_n} + \beta m(x_n) & \mbox{ for } x_n\in
B_\delta(z_0)\setminus \bar B_{\delta - \epsilon}(z_0) \vspace{2mm} \\
\delta_{x_n} & \mbox{ for } x_n\in\bar B_{\delta-\epsilon}(z_0),
\end{array}\right.
\end{split}
\end{equation*}
where the probability $m(x)$ is uniform in the set $B_\epsilon(x)
\cap Y$ and is given by: 
\begin{equation}\label{em}
m(x) = \frac{\mathcal{L}_N\lfloor \big(B_\epsilon(x) \cap
  Y\big)}{|B_\epsilon(x) \cap Y|}.
\end{equation} 
Let now $\bar{\mathbb{P}}^{x_0}_{\bar\sigma_I, \bar\sigma_{0, II}}$ and 
$\bar{\mathbb{P}}^{n, x_0}_{\bar\sigma_I, \bar\sigma_{0, II}}$ be the
Borel probability measures on $Y^{\infty, x_0}$ defined as in 
subsection \S \ref{probab}. As in the proof of Lemma \ref{lemkoniec},
observe that:
\begin{equation}\label{koniec2}
\bar{\mathbb{P}}^{x_0}_{\bar\sigma_I, \bar\sigma_{0, II}} \big(\left\{\bar\tau_0 <
  +\infty\right\}\big) = 1,
\end{equation}
so that, in addition to (\ref{koniec}) there also holds:
${\mathbb{P}}^{x_0}_{\bar\sigma_I, \bar\sigma_{0, II}}
\big(\{\bar\tau <  +\infty\}\big) = 1.$ 

In view of   Lemma \ref{transitionlemma} we hence observe:
\begin{equation}\label{kot}
\begin{split}
\mathbb{E}^{x_0}_{\tau, \sigma_{I}, \sigma_{0, II}} [\tau] & =
\sum_{n=1}^\infty n \mathbb{P}^{n, x_0}_{\tau, \sigma_I, \sigma_{0,
    II}} \big(\omega\in X^{x_0,\infty}; ~ \tau(\omega)=n\big) \\ & = 
\sum_{n=1}^\infty n \bar{\mathbb{P}}^{n, x_0}_{\bar{\sigma}_I, \bar{\sigma}_{0,
    II}} \big(\omega\in X^{x_0,\infty}; ~ \tau(\omega)=n\big) \\ & \leq 
\sum_{n=1}^\infty n \bar{\mathbb{P}}^{n, x_0}_{\bar{\sigma}_I, \bar{\sigma}_{0,
    II}} \big(\omega\in Y^{x_0,\infty}; ~ \bar\tau(\omega)=n\big) \\ &
= \bar{\mathbb{E}}^{x_0}_{\bar{\sigma}_{I}, \bar{\sigma}_{0, II}}
[\bar\tau] \leq \bar{\mathbb{E}}^{x_0}_{\bar{\sigma}_{I},
  \bar{\sigma}_{0, II}} [\bar\tau_0].
\end{split}
\end{equation}
Indeed, if $\omega = (x_0, x_1, \ldots)\in X^{\infty, x_0}$ satisfies $\tau(\omega) = n$, then
for all $k<n$ we have: $x_k\in\Omega$ and $(x_0, x_1,\ldots
x_k)\not\in A^\tau_k$, and hence there must be: $\gamma_k[x_0,\ldots x_k]
= \bar \gamma_k[x_0,\ldots x_k]$. Consequently:
$$\mathbb{P}^{n, x_0}_{\tau, \sigma_I, \sigma_{0,
    II}} \big(\omega\in X^{x_0,\infty}; ~ \tau(\omega)=n\big) =
\bar{\mathbb{P}}^{n, x_0}_{\bar{\sigma}_I, \bar{\sigma}_{0,  II}}
\big(\omega\in X^{x_0,\infty}; ~ \tau(\omega)=n\big).$$ 

\medskip

{\bf 3.} We now estimate the
expectation $\bar{\mathbb{E}}^{x_0}_{\bar\sigma_I, \bar\sigma_{0,
      II}}[\bar\tau_0]$. Let $v_0:(0,+\infty)\to\mathbb{R}$ be a
  smooth, increasing and concave function of the form:
$$v_0(s) = \left\{\begin{array}{ll} -as^2 - bs^{2-N} + c & \mbox{ for
    } N>2\\ -as^2 -b\log s + c & \mbox{ for } N=2,\end{array}\right.$$
where the positive constants $a, b, c$ are such that the function
$v(x) = v_0(|x-z_0|)$ solves the following problem:
$$\left\{\begin{array}{ll} \Delta v = -2(N+2) & \mbox{ in
    } B_R(z_0)\setminus \bar B_\delta(z_0)\\ v=0 & \mbox{ on
    } \partial B_\delta(z_0) \\ \frac{\partial v}{\partial\vec n} = 0
    & \mbox{ on } \partial B_R(z_0).\end{array}\right.$$
As in (\ref{aiuto}), we obtain:
\begin{equation}\label{aiuto3}
\forall x\in \mathbb{R}^N\setminus \bar B_{\delta-\epsilon}(z_0) \qquad
\fint_{B_\epsilon(x)} v(w)~\mbox{d}w = v(x) - \epsilon^2 .
\end{equation}
On the other hand, for every $x\in B_R(z_0)\setminus \bar
B_{\delta-\epsilon}(z_0)$, we have:
\begin{equation*}
\begin{split}
& \fint_{B_\epsilon(x) \cap B_R(z_0)} v - \fint_{B_\epsilon(x)} v \\ & =
\left(\frac{1}{|B_\epsilon(x) \cap B_R(z_0)|} -  \frac{1}{|B_\epsilon(x) |} \right) \int_{B_\epsilon(x) \cap
  B_R(z_0)} v - \frac{1}{|B_\epsilon(x) |}
\int_{B_\epsilon(x)\setminus B_R(z_0)} v \\ & \leq \left(\frac{1}{|B_\epsilon(x) \cap B_R(z_0)|} -
  \frac{1}{|B_\epsilon(x) |} \right) v_0(R) |B_\epsilon(x) \cap
  B_R(z_0)|- \frac{1}{|B_\epsilon(x) |} v_0(R) |B_\epsilon(x)\setminus
  B_R(z_0)| = 0.
\end{split}
\end{equation*}
Consequently, recalling (\ref{em}) and in view of (\ref{aiuto3}):
\begin{equation}\label{aiuto2}
\forall x\in Y\setminus \bar B_{\delta-\epsilon}(z_0) \qquad \int
v~\mbox{d}m(x) \leq v(x) -\epsilon^2.
\end{equation}

Consider the following bounded Borel functions on $Y$:
\begin{equation*}
Q_n(x) =\left\{\begin{array}{ll} v(x) + \frac{\beta}{2}n\epsilon^2 &
    \mbox{ if } |x-z_0|>\delta-\epsilon\\
v(x) & \mbox{ if } \delta-2\epsilon < |x-z_0|\leq \delta-\epsilon\\
v_0(\delta-2\epsilon) & \mbox{ if } |x-z_0|\leq \delta-2\epsilon,
\end{array}\right.
\end{equation*}
and compute the conditional expectation of the random variables
$Q_n\circ x_n$, which by Lemma \ref{condi} equals: 
$$\bar{\mathbb{E}}^{y_0}_{\bar\sigma_I, \bar\sigma_{0,II}} \big\{
Q_n\circ x_n| \mathcal{F}_{n-1}^{x_0}\big\} (x_0,\ldots x_{n-1}) =
\int_Y Q_n ~\mbox{d}\bar\gamma_{n-1}[x_0,\ldots x_{n-1}].$$
We distinguish three cases.

\medskip

\noindent {\underline{{\em Case 1:} $x_{n-1}\in Y\setminus \bar B_\delta(z_0)$.}} Using the
fact that $v_0$ is increasing, denoting its Lipschitz constant on
$[\delta/3, R+\delta]$ by $C_\delta$, recalling (\ref{aiuto2}) and
observing that $|\bar\sigma_{II}(x_{n-1}) -
z_0|>\delta-\epsilon$, we get:
\begin{equation*}
\begin{split}
\int_Y &Q_n  ~\mbox{d}\bar\gamma_{n-1}[x_0,\ldots x_{n-1}] =
\frac{\alpha}{2} Q_n\big(\bar\sigma_I^{n-1}(x_0,\ldots x_{n-1})\big) +
\frac{\alpha}{2} Q_n\big(\bar\sigma_{0, II}^{n-1}(x_{n-1})\big) +
\beta\int Q_n~\mbox{d}m(x_{n-1}) \\ & \leq \frac{\alpha}{2}
v_0\big(|x_{n-1} - z_0| + \epsilon\big) +  \frac{\alpha}{2}
v_0\big(|x_{n-1} - z_0| - \epsilon + \epsilon^3\big) + \beta
\big(v_0(|x_{n-1} - z_0|)-\epsilon^2\big) +
\frac{\delta}{2}n\epsilon^2 \\ & \leq \alpha v_0( |x_{n-1} - z_0|) + \beta
v_0 (|x_{n-1} - z_0|) + C_\delta\epsilon^3 -\beta\epsilon^2
+\frac{\beta}{2}n\epsilon^2 \\ & \leq v_0(|x_{n-1} - z_0|) +
\frac{\beta}{2}(n-1)\epsilon^2 = Q_{n-1}(x_{n-1}),
\end{split}
\end{equation*}
where the last two inequalities follow from concavity of $v_0$, and the
fact that if only:
\begin{equation}\label{restri}
\epsilon < \min\big(\frac{\beta}{2C_\delta}, \frac{\delta}{3}\big),
\end{equation}
then: $ C_\delta\epsilon^3 -\beta\epsilon^2 + \frac{\beta}{2}n\epsilon^2 \leq
\frac{\beta}{2}\epsilon^2(n-1).$

\medskip

\noindent {\underline{{\em Case 2:} $x_{n-1}\in B_\delta(z_0)\setminus \bar B_{\delta-\epsilon}(z_0)$.}} Then:
\begin{equation*}
\begin{split}
\int_Y Q_n  ~\mbox{d}\bar\gamma_{n-1}[x_0,\ldots x_{n-1}] & =
{\alpha}Q_n( x_{n-1}) + \beta\int Q_n~\mbox{d}m(x_{n-1}) \\ & \leq \alpha
v_0(|x_{n-1} - z_0|) + \beta
\big(v_0(|x_{n-1} - z_0|)-\epsilon^2\big) +
\frac{\delta}{2}n\epsilon^2 \\ & = v_0( |x_{n-1} - z_0|) + 
\frac{\beta}{2}(n-1)\epsilon^2 = Q_{n-1}(x_{n-1}).
\end{split}
\end{equation*}

\medskip

\noindent {\underline{{\em Case 3:} $x_{n-1}\in \bar
    B_{\delta-\epsilon}(z_0)$.}} In this case,  we directly have:
\begin{equation*}
\int_Y Q_n  ~\mbox{d}\bar\gamma_{n-1}[x_0,\ldots x_{n-1}]  = Q_{n-1}(x_{n-1}).
\end{equation*}

\medskip

Consequently, it follows that $\{Q_n\circ x_n\}_{n\geq 0}$ is a
supermartingale with respect to the filtration
$\{\mathcal{F}^{x_0}_n\}$, under the assumption (\ref{restri}).
Applying Doob's theorem to pairs of bounded stopping times:
$\bar\tau_0\wedge n$ and $ 0$, we obtain:
$$v_0(|x_0 - z_0|) \geq \bar{\mathbb{E}}^{x_0}_{\bar\sigma_I,
  \bar\sigma_{0, II}} [Q_{\bar\tau_0\wedge n}] = \bar{\mathbb{E}}^{x_0}_{\bar\sigma_I,
  \bar\sigma_{0, II}} [v_0(|x_{\bar\tau_0\wedge n} - z_0|)] +
\frac{\beta}{2}\epsilon^2 \bar{\mathbb{E}}^{x_0}_{\bar\sigma_I,
  \bar\sigma_{0, II}} [\bar\tau_0\wedge n].  $$
After passing with $n\to\infty$, in view of the monotone convergence
and dominated convergence theorems, and applying \ref{koniec2} we get:
\begin{equation}\label{ges}
\frac{\beta}{2}\epsilon^2 \bar{\mathbb{E}}^{x_0}_{\bar\sigma_I,
  \bar\sigma_{0, II}} [\bar\tau_0] \leq v_0(|x_0 - z_0|)  + \bar{\mathbb{E}}^{x_0}_{\bar\sigma_I,
  \bar\sigma_{0, II}} [|v_0(|x_{\bar\tau_0} - z_0|)|].
\end{equation}
Since $v_0(\delta)=0$, it follows that $v_0(|x_0-z_0|) \leq C_\delta
\big(|x_0 - z_0|-\delta\big) = C_\delta |x_0-y_0|$. Further, whenever
$\bar\tau_0(\omega)<+\infty$, we have: $|v_0(|x_{\bar\tau_0} -
z_0|)|\leq C_\delta\epsilon$. Now, (\ref{ges}) together with
(\ref{kot}) and (\ref{pies}) imply:
$$\mathbb{E}^{x_0}_{\tau, \sigma_I, \sigma_{0, II}} [|x_\tau - y_0|]
\leq C\delta + C_\delta (|x_0 - x_0|+\epsilon)$$
for all $\epsilon$ sufficiently small  (as in (\ref{restri})).

\medskip

{\bf 4.} Clearly, exchanging the roles of $\sigma_I$ and
$\sigma_{II}$, and defining $\sigma_{0, I}$ by means of (\ref{sn})
while setting $\sigma_{II}$ to be fixed, the same proof as above yields:
$$\mathbb{E}^{x_0}_{\tau, \sigma_{I, 0}, \sigma_{II}} [|x_\tau - x_0|]
\leq C\delta + C_\delta (|x_0 - y_0|+\epsilon)$$
for all $\epsilon$ sufficiently small. 
This gives (\ref{three}) and ends the proof of Lemma \ref{closetoboundary}.
\endproof

\section{Appendix: a proof of Lemma \ref{lemkoniec}: games end almost-surely}

{\bf 1.} Consider a new ``game-board'' $Y=\mathbb{R}^N$ with the same
initial token position $x_0\in \Omega$. By the same symbols $\sigma_I$
and $\sigma_{II}$ we denote
the extensions on $\{Y^{n}\}_{n=0}^\infty$ of the given strategies $\sigma_I$
and $\sigma_{II}$, defined as in the formula (\ref{easy_ext}), where in order to simplify
notation we suppress the overline in $\bar\sigma$. Define also the new
transition probabilities:
$$\gamma_n[x_0,\ldots, x_n] =
\frac{\alpha}{2}\delta_{\sigma_I^n(x_0,\ldots, x_n)} +
\frac{\alpha}{2}\delta_{\sigma_{II}^n(x_0,\ldots, x_n)} +
\frac{\beta}{|B_\epsilon|} \mathcal{L}_N\lfloor B_\epsilon(x_n), $$ 
and let $\mathbb{P}^{n, x_0}_{\sigma_I, \sigma_{II}}$ and
$\mathbb{P}^{x_0}_{\sigma_I, \sigma_{II}}$ be the resulting
probability measures on $Y^{\infty, x_0}$ as in subsection \S \ref{probab}.
By Lemma \ref{transitionlemma} and since $\tau\leq \tau_0$, it follows that:
\begin{equation}\label{app1}
\begin{split}
\mathbb{P}^{x_0}_{\tau, \sigma_{I}, \sigma_{II}}& (\{\tau<\infty\})  =
\sum_{n=0}^\infty \mathbb{P}^{n, x_0}_{\tau, \sigma_I, \sigma_{II}}
\big(\{\omega\in X^{x_0,\infty}; ~ \tau(\omega)=n\big\}) \\ & =  
\sum_{n=0}^\infty \bar{\mathbb{P}}^{n, x_0}_{\bar{\sigma}_I,
  \bar{\sigma}_{II}} \big(\{\omega\in X^{x_0,\infty}; ~ \tau(\omega)=n\}\big) =
\sum_{n=0}^\infty \bar{\mathbb{P}}^{n, x_0}_{\bar{\sigma}_I,
  \bar{\sigma}_{II}} \big(\{\omega\in Y^{x_0,\infty}; ~
\tau(\omega)=n\}\big) \\ & 
= {\mathbb{P}}^{x_0}_{\bar{\sigma}_{I}, \bar{\sigma}_{II}}
(\{\tau<\infty\}) \geq {\mathbb{P}}^{x_0}_{\bar{\sigma}_{I},
  \bar{\sigma}_{II}} (\{\tau_0<\infty\}).
\end{split}
\end{equation}
Let now $A_0$ be the sector in $B_\epsilon=B_\epsilon (0)$:
$$A_0=\Big\{x\in\mathbb{R}^N; ~ |x|\in(\epsilon/2, \epsilon) ~~\mbox{ and }~~
\angle (x, e_1) \in (-\pi/8, \pi/8)\Big\}. $$
For $M\in \mathbb{N}$ sufficiently large to ensure that $M$
consecutive shifts of the token by vectors chosen from $A_0$ will get
the token, originally located at any point in $\Omega$, out of
$\Omega$, define:
$$S_{x_0} = \Big\{\omega=(x_0, x_1,\ldots)\in Y^{x_0, \infty}; ~
\exists i_0\quad\forall i=i_0, i_0+1,\ldots, i_0+M \quad x_{i+1} -
x_i\in A_0\Big\}.$$
t is clear that:
\begin{equation}\label{app2}
\mathbb{P}^{x_0}_{\sigma_I, \sigma_{II}} (\{\tau_0<\infty\}) \geq \mathbb{P}^{x_0}_{\sigma_I, \sigma_{II}} (S_{x_0}). 
\end{equation}

\medskip

{\bf 2.} We now show that the probability in the right hand side of
(\ref{app2}) equals $1$. Recall that for a bounded
$\mathcal{F}^{x_0}$-measurable function
$f:Y^{x_0,\infty}\to\mathbb{R}$, its conditional expectation
$\mathbb{E}^{x_0}_{\sigma_I, \sigma_{II}}\{f|\mathcal{F}_1^{x_0}\}$ is
the function: $(x_0, x_1)\mapsto \mathbb{E}^{x_1}_{\sigma_I',
  \sigma_{II}'}[f']$, where $\sigma_I'$, $\sigma_{II}'$ are strategies
on $Y^{x_0, \infty}$ given by: 
$$(\sigma_I')^n(x_1,\ldots, x_{n+1}) =
\sigma_I^{n+1}(x_0, x_1,\ldots, x_{n+1}), \qquad
(\sigma_{II}')^n(x_1,\ldots, x_{n+1}) = \sigma_{II}^{n+1}(x_0, x_1,\ldots, x_{n+1}),$$ 
while the Borel random
variable $f':Y^{x_1, \infty}\to\mathbb{R}$ is similarly set to be:
$f'(x_1,x_2,\ldots ) = f( x_0, x_1,x_2,\ldots )$.
Consequently:
\begin{equation}\label{app3}
\begin{split}
\mathbb{P}^{x_0}_{\sigma_{I}, \sigma_{II}}& (S_{x_0})  =
\int_{\mathbb{R}^N} \mathbb{E}^{x_1}_{\sigma_I',
  \sigma_{II}'}\big[(\chi_{S_{x_0}})'\big]~\mbox{d}\gamma_0[x_0] \\ & = 
\frac{\alpha}{2}\mathbb{E}^{\sigma_I^0(x_0)}_{\sigma_I',
  \sigma_{II}'}\big[(\chi_{S_{x_0}})'\big] + \frac{\alpha}{2}\mathbb{E}^{\sigma_{II}^0(x_0)}_{\sigma_I',
  \sigma_{II}'}\big[(\chi_{S_{x_0}})'\big] +
\beta\fint_{B_\epsilon(x_0)} \mathbb{E}^{x_1}_{\sigma_I',
  \sigma_{II}'}\big[(\chi_{S_{x_0}})'\big] ~\mbox{d}x_1 \\ &
=  \frac{\alpha}{2}\mathbb{P}^{\sigma_I^0(x_0)}_{\sigma_I',
  \sigma_{II}'}(S^{x_0}_{\sigma_I^0(x_0)}) + \frac{\alpha}{2}\mathbb{P}^{\sigma_{II}^0(x_0)}_{\sigma_I',
  \sigma_{II}'}(S^{x_0}_{\sigma_{II}^0(x_0)}) +
\frac{\beta}{|B_\epsilon|} \int_{B_\epsilon(x_0)} \mathbb{P}^{x_1}_{\sigma_I', \sigma_{II}'}(S^{x_0}_{x_1})~\mbox{d}x_1,
\end{split}
\end{equation}
where each set $S_{x_1}^{x_0} = \{(x_1,x_2,\ldots)\in Y^{x_1,\infty};~ (x_0,
x_1,x_2,\ldots)\in S_{x_0}\}$ clearly contains $S_{x_1}$. Let now:
\begin{equation}\label{defq}
q(x) = \inf_{\tilde\sigma_I,\tilde\sigma_{II}}
\mathbb{P}^x_{\tilde\sigma_I,\tilde\sigma_{II}} (S_x).
\end{equation}
By an easy translation invariance argument, $q(x) = q$ is actually
independent of $x\in\mathbb{R}^N$. Hence, in view of (\ref{app3}) we
obtain:
\begin{equation}\label{app4}
\begin{split}
\mathbb{P}^{x_0}_{\sigma_{I}, \sigma_{II}}  (S_{x_0})  & \geq 
\frac{\alpha}{2}\mathbb{P}^{\sigma_I^0(x_0)}_{\sigma_I',
  \sigma_{II}'}(S_{\sigma_I^0(x_0)}) + \frac{\alpha}{2}\mathbb{P}^{\sigma_{II}^0(x_0)}_{\sigma_I',
  \sigma_{II}'}(S_{\sigma_{II}^0(x_0)}) \\ & \qquad\quad +
\frac{\beta}{|B_\epsilon|} \int_{B_\epsilon(x_0)\setminus (x_0+A_0)}
\mathbb{P}^{x_1}_{\sigma_I', \sigma_{II}'}(S_{x_1})~\mbox{d}x_1 +
\frac{\beta}{|B_\epsilon|} \int_{x_0+A_0} 
\mathbb{P}^{x_1}_{\sigma_I', \sigma_{II}'}(S^{x_0}_{x_1})~\mbox{d}x_1
\\ & \geq \Big(\frac{\alpha}{2} + \frac{\alpha}{2} +
\frac{\beta}{|B_\epsilon|} |B_\epsilon\setminus A_0|\Big)q + 
\frac{\beta}{|B_\epsilon|} \int_{x_0+A_0} 
\mathbb{P}^{x_1}_{\sigma_I', \sigma_{II}'}(S^{x_0}_{x_1})~\mbox{d}x_1
\\ & = \theta q + \frac{\beta}{|B_\epsilon|} \int_{x_0+A_0} 
\mathbb{P}^{x_1}_{\sigma_I', \sigma_{II}'}(S^{x_0}_{x_1})~\mbox{d}x_1,
\end{split}
\end{equation}
where we defined:
$$\theta = \alpha + \beta\Big(1-\frac{|A_0|}{|B_\epsilon|}\Big).$$

\medskip

{\bf 3.} Similarly as in the previous step, for every $x_1\in x_0+A_0$ there holds:
\begin{equation*}
\begin{split}
\mathbb{P}^{x_1}_{\sigma_{I}', \sigma_{II}'}  (S^{x_0}_{x_1})  =
\frac{\alpha}{2}\mathbb{P}^{\sigma_I^1(x_0, x_1)}_{\sigma_I'',
  \sigma_{II}''}(S^{x_0, x_1}_{\sigma_I^1(x_0, x_1)}) & + \frac{\alpha}{2}\mathbb{P}^{\sigma_{II}^1(x_0, x_1)}_{\sigma_I'',
  \sigma_{II}''}(S^{x_0, x_1}_{\sigma_{II}^1(x_0, x_1)}) \\ & +
\frac{\beta}{|B_\epsilon|} \int_{B_\epsilon(x_0)}
\mathbb{P}^{x_2}_{\sigma_I'', \sigma_{II}''}(S^{x_0,x_1}_{x_2})~\mbox{d}x_2,
\end{split}
\end{equation*}
where the set $S^{x_0,x_1}_{x_2}=\{(x_2, x_3,\ldots)\in Y^{x_2,\infty}; ~
(x_2, x_2, x_2, x_3\ldots)\in S_{x_0}\}$ contains the set
$S_{x_2}$. By (\ref{defq}) we see that:
$$\inf_{\tilde\sigma_I,\tilde\sigma_{II}}
\mathbb{P}^{x_2}_{\tilde\sigma_I,\tilde\sigma_{II}} (S^{x_0,  x_1}_{x_2})\geq q,$$
and hence:
\begin{equation*}
\mathbb{P}^{x_1}_{\sigma_{I}', \sigma_{II}'}  (S^{x_0}_{x_1})  \geq
\theta q + \frac{\beta}{|B_\epsilon|} \int_{x_1+ A_0}
\mathbb{P}^{x_2}_{\sigma_I'', \sigma_{II}''}(S^{x_0,x_1}_{x_2})~\mbox{d}x_2.
\end{equation*}
Since $1-\theta = \beta |A_0|/|B_\epsilon|$, the estimate in
(\ref{app4}) becomes:
\begin{equation*}
\mathbb{P}^{x_0}_{\sigma_{I}, \sigma_{II}}  (S_{x_0})   = \theta q +
(1-\theta)\theta q +
\Big(\frac{\beta}{|B_\epsilon|} \Big)^2\int_{x_0+A_0} \int_{x_1+A_0} 
\mathbb{P}^{x_2}_{\sigma_I'', \sigma_{II}''}(S^{x_0,x_1}_{x_2})~\mbox{d}x_2 ~\mbox{d}x_1.
\end{equation*}

Iterating the same argument as above $M$ times, we arrive at:
\begin{equation}\label{app5}
\begin{split}
\mathbb{P}^{x_0}_{\sigma_{I}, \sigma_{II}}  (S_{x_0})   \geq & ~\theta q +
(1-\theta)\theta q + (1-\theta)^2\theta q + \ldots +
(1-\theta)^{M-1}\theta q \\ & + 
\Big(\frac{\beta}{|B_\epsilon|} \Big)^M\int_{x_0+A_0} \int_{x_1+A_0}
\ldots \int_{x_{M-1}+A_0} 
\mathbb{P}^{x_M}_{\sigma_I^{'M}, \sigma_{II}^{'M}}(S^{x_0,\ldots,x_{M-1}}_{x_M})~\mbox{d}x_M\ldots ~\mbox{d}x_1.
\end{split}
\end{equation}
But each probability under the iterated integrals equals to $1$,
because: $S^{x_0,\ldots,x_{M-1}}_{x_M} = Y^{x_M,\infty}$ for $x_1\in
x_0+A_0$, $x_2\in x_1+A_0$, $\ldots, x_M\in
x_{M-1}+A_0$. Consequently, by (\ref{app5}) we get:
$$\mathbb{P}^{x_0}_{\sigma_{I}, \sigma_{II}}  (S_{x_0}) \geq
\sum_{n=0}^{M-1}(1-\theta)^n\theta q + (1-\theta)^M =
\big(1-(1-\theta)^M\big) q + (1-\theta)^M.$$
Infimizing over all strategies $\sigma_I, \sigma_{II}$, it follows
that $q\geq 1$, since $\theta<1$ because of $\beta>0$. Further: 
$$\mathbb{P}^{x_0}_{\sigma_{I}, \sigma_{II}}  (S_{x_0}) \geq q=1.$$
This achieves (\ref{koniec}) in view of (\ref{app1}) and (\ref{app2}).\endproof

\section{Appendix: a proof of Lemma \ref{uni}: uniqueness of viscosity
solutions}\label{sceuni}

{\bf 1.}  Firstly, note that the continuous function $u$ is a
viscosity $p$-supersolution to (\ref{dirichlet}). Thus, by the
classical result in \cite{JLM}, $u$ is $p$-superharmonic in $\Omega$,
and consequently (see \cite{L}) we have $u\in W^{1,p}_{loc}(\Omega)$.
In the same manner, it follows from Definition \ref{viscosity} that 
$u$ is a viscosity $p$-subsolution on the open set
$\mathcal{V}=\{x\in\Omega; ~ u(x)>\Psi(x)\}$, 
hence $u$ is $p$-subharmonic in $\mathcal{V}$.


Therefore,  using the variational definitions of $p$-super- and $p$-subharmonic functions, we have that  for any
open, Lipschitz domain $\mathcal{U}\subset\subset \Omega$ there holds:
\begin{equation}\label{lu1}
\int_{\mathcal{U}}|\nabla u|^p\leq \int_{\mathcal{U}}|\nabla
(u+\phi)|^p \qquad \forall \phi\in\mathcal{C}_0^\infty(\mathcal{U},\mathbb{R}_+),
\end{equation}
\begin{equation}\label{lu2}
\int_{\mathcal{U}\cap\mathcal{V}}|\nabla u|^p\leq \int_{\mathcal{U}\cap\mathcal{V}}|\nabla
(u+\phi)|^p \qquad \forall \phi\in\mathcal{C}_0^\infty(\mathcal{U}\cap\mathcal{V},\mathbb{R}_-).
\end{equation}

Let now $\phi\in\mathcal{C}_0^\infty(\mathcal{U},\mathbb{R})$ be such
that $\Psi\leq u+\phi$. We write: $\phi = \phi^+ +\phi^-$
as the sum of the positive and negative parts of $\phi$. Denote:
\begin{equation*}
D^+ = \{x\in\mathcal{U}; ~ \phi(x)>0\}\quad\mbox{
  and } \quad D^- = \{x\in\mathcal{U}; ~ \phi(x)<0\}\subset\mathcal{V}.
\end{equation*}
Then we have, in view of (\ref{lu1}) and (\ref{lu2}):
\begin{equation}\label{lu3}
\begin{split}
 \int_{\mathcal{U}}|\nabla u +\nabla \phi|^p  & = \int_{D^+}|\nabla u
+\nabla \phi|^p + \int_{D^-}|\nabla u +\nabla \phi|^p +
\int_{\{\phi=0\}}|\nabla u|^p \\ & =
\int_{D^+}|\nabla u +\nabla (\phi^+)|^p + \int_{\mathcal{U}\cap\mathcal{V}}|\nabla
u +\nabla (\phi^-)|^p - \int_{(\mathcal{U}\cap\mathcal{V})\setminus D^-}|\nabla u|^p  
\\ & \geq \int_{D^+}|\nabla u|^p 
+ \int_{\mathcal{U}\cap\mathcal{V}}|\nabla u|^p - \int_{(\mathcal{U}\cap\mathcal{V})\setminus D^-}|\nabla u|^p 
 = \int_{\mathcal{U}}|\nabla u|^p.
\end{split}
\end{equation}
The above means precisely that $u$ is a variational solution of the
obstacle problem on $\mathcal{U}$, with the lower obstacle $\Psi$ and
boundary data $f=u_{|\partial\mathcal{U}}$; we denote this problem by
$\mathcal{K}_{\Psi, f}(\mathcal{U})$. 
Existence and uniqueness of such variational solution is an easy direct consequence of the strict convexity of
the functional $\int_{\mathcal{U}} |\nabla u|^p$. It is also quite
classical that such solutions obey a comparison principle \cite{L}.

\medskip

{\bf 2.} Let now $u$ and $\bar u$ be as in the statement of the Lemma.
Fix $\epsilon>0$. By the uniform continuity of
$u$, $\bar u$ on $\Omega$ and the fact that they coincide on
$\partial\Omega$, there exists $\delta>0$ such that: 
\begin{equation}\label{lu4}
|u(x)-\bar u(x)|\leq \epsilon \qquad\forall x\in\mathcal{O}_\delta(\partial\Omega)
:= \big(\partial\Omega+B(0,\delta)\big)\cap\bar\Omega.
\end{equation}
Consider an open, Lipschitz set  $\mathcal{U}$  satisfying:
$\Omega\setminus \mathcal{O}_\delta(\partial\Omega)\subset\subset \mathcal{U}
\subset\subset \Omega.$
By the argument in Step 1, $u$ is the variational solution of the
problem set $\mathcal{K}_{\Psi,  u_{\mid \partial\mathcal{U}}}(\mathcal{U})$, and it is also easy to
observe that $\bar u+\epsilon$
is the variational solution of the problem $\mathcal{K}_{\Psi,
 \bar{u}_{\mid\partial\mathcal{U}}+\epsilon}(\mathcal{U})$.
Since $u<\bar u+\epsilon$ on $\partial\mathcal{U}$ in view of
(\ref{lu4}), the comparison principle implies that
$u\leq\bar u+\epsilon $ in $\bar{\mathcal{U}}$. 

Reversing the same  argument and taking into account (\ref{lu4}), we arrive at:
$$|u(x) - \bar u(x)|\leq \epsilon \qquad \forall x\in\bar\Omega.$$
We conclude that $u=\bar u$ in $\bar\Omega$ by passing to the limit
$\epsilon\to 0$ in the above bound.
\endproof

\end{document}